\documentclass{amsart}

\usepackage{graphicx}
\usepackage{a4wide}
\usepackage{psfrag}
\usepackage[all]{xy}

\let\cal\mathcal
\def\AA{{\cal A}}

\def\CC{{\cal C}}
\def\DD{{\cal D}}

\def\HH{{\cal H}}

\def\LL{{\cal L}}

\def\PP{{\cal P}}
\def\QQ{{\cal Q}}
\def\RR{{\cal R}}
\def\SS{{\cal S}}
\def\TT{{\cal T}}
\def\UU{{\cal U}}

\def\ZZ{{\cal Z}}

\let\blb\mathbb

\def\bZ{{\blb Z}}

\def\bN{{\blb N}}

\def\bS{{\blb S}}

\def\bZ{{\blb Z}}

\let\frak\mathfrak

\def\aa{\frak{a}}

\def\Mod{\operatorname{Mod}}

\def\mod{\operatorname{mod}}

\def\rad{\operatorname {rad}}

\def\rep{\operatorname{rep}}

\def\Ext{\operatorname {Ext}}

\def\Hom{\operatorname {Hom}}

\def\End{\operatorname {End}}
\def\RHom{\operatorname {RHom}}

\def\Sl{\operatorname {Sl}}

\def\im{\operatorname {im}}

\def\cone{\operatorname {cone}}
\def\coker{\operatorname {coker}}
\def\ker{\operatorname {ker}}

\def\End{\operatorname {End}}

\def\r{\rightarrow}

\def\d{\downarrow}

\DeclareMathOperator{\Ind}{Ind}

\DeclareMathOperator{\ind}{ind}

\newcommand\Db{D^{b}}

\newcommand\tri[3]{#1\to #2\to #3\to #1[1]}

\renewcommand\r{{d^{\bullet}}}

\renewcommand\d{d}

\newcommand\Sr{S^{\bullet}}
\renewcommand\Sl{S_{\bullet}}
\newcommand\SrQQ{S^{\bullet}_\QQ}
\newcommand\SlQQ{S_{\bullet}^\QQ}

\newcommand\TLA{T^A_\bullet}

\newcommand\tm{\tau^{-}}
\renewcommand\t{\tau}

\newtheorem{lemma}{Lemma}[section]
\newtheorem{proposition}[lemma]{Proposition}
\newtheorem{theorem}[lemma]{Theorem}
\newtheorem{corollary}[lemma]{Corollary}

\theoremstyle{definition}

\newtheorem{example}[lemma]{Example}
\newtheorem{definition}[lemma]{Definition}

\newtheorem{construction}[lemma]{Construction}

\newtheorem{observation}[lemma]{Observation}

{

}

\theoremstyle{remark}

\newtheorem{remark}[lemma]{Remark}

\newdimen\uboxsep \uboxsep=1ex
\def\uboxn#1{\vtop to 0pt{\hrule height 0pt depth 0pt\vskip\uboxsep
\hbox to 0pt{\hss #1\hss}\vss}}

\def\uboxs#1{\vbox to 0pt{\vss\hbox to 0pt{\hss #1\hss}
\vskip\uboxsep\hrule height 0pt depth 0pt}}

\def\Ob{\operatorname{Ob}}

\def\cofin{\operatorname{cofin}}
\def\coinit{\operatorname{coinit}}

\newcommand\exa{\nopagebreak \begin{center}\smallskip \nopagebreak               \begin{minipage}[t]{.4\textwidth}\sloppy}
\newcommand\exb{\end{minipage} \hspace{.05\textwidth} \begin{minipage}[t]{.4\textwidth}\sloppy}
\newcommand\exc{\end{minipage} \smallskip\end{center}}

\title{Hereditary Categories with Serre Duality which are generated by Preprojectives}
\author{Carl Fredrik Berg}
\author{Adam-Christiaan van Roosmalen}
\address{Carl Fredrik Berg\\Institutt for matematiske fag\\
NTNU\\7491 Trondheim\\Norway\\
(Currently working for StatoilHydro R\&D Centre\\Arkitekt Ebbells veg 10\\Rotvoll\\7053 Ranheim\\Norway)} \email{carlpaatur@hotmail.com}
\address{Adam-Christiaan van Roosmalen\\Mathematisches Institut\\Universit\"{a}t Bonn
\\Endenicher Allee 60\\53111 Bonn\\Germany}\email{vroosmal@math.uni-bonn.de}


\begin{document}

\bibliographystyle{amsplain}

\begin{abstract}
We show that every $k$-linear abelian Ext-finite hereditary category with Serre duality which is generated by preprojective objects is derived equivalent to the category of representations of a strongly locally finite thread quiver.
\end{abstract}

\maketitle

\tableofcontents
\section{Introduction}

Throughout, let $k$ be an algebraically closed field.  In \cite{ReVdB02}, Reiten and Van den Bergh classify $k$-linear abelian hereditary Ext-finite noetherian categories with Serre duality.  One result in there is that every such category is a direct sum of a category without nonzero projectives, and a category generated by preprojective objects.  The latter categories were of specific interest since there was no known way to relate them --through equivalences or derived equivalences-- to known abelian categories.  Reiten and Van den Bergh gave a construction by formally inverting the (right) Serre functor, and in \cite{Ringel02b} Ringel gave a construction using ray quivers.  (In \cite{BergVanRoosmalen08} it was shown that these categories were derived equivalent to representations of strongly locally finite quivers, i.e. quivers whose indecomposable projective and injective representations have finite length.)

Reiten and Van den Bergh asked whether every hereditary categories with Serre duality is derived equivalent to a noetherian one, and thus fit --up to derived equivalence-- into their classification.  In \cite{Ringel02} however, Ringel gave a class of counterexamples.  Reiten then asks in \cite{Reiten02} whether it is feasible to have a classification of hereditary categories with Serre duality which are generated by preprojectives, but not necessarily noetherian.

This paper is the third paper of the authors to answer this question (the other two being \cite{BergVanRoosmalen08, BergVanRoosmalen09}); we provide an answer to this question up to derived equivalence in terms of representations of thread quivers (see below):

\begin{theorem}\label{theorem:Introduction}
Let $\AA$ be a $k$-linear abelian hereditary Ext-finite category with Serre duality which is generated by preprojective objects.  Then $\Db \AA \cong \Db \rep_k Q$ where $Q$ is a strongly locally finite thread quiver.
\end{theorem}

The undefined concepts in this theorem will be introduced below.  Roughly speaking a thread quiver is a (possibly infinite) quiver where some of the arrows have been replaced by locally discrete (=without accumulation points) linearly ordered set.  Strong local finiteness is an additional finteness property ensuring that the category of finitely presented representations has Serre duality.

The proof of this theorem consists out of two steps.  In the first step (up to and including \S\ref{section:Description}) we prove a version of Theorem \ref{theorem:Introduction} under an additional assumption, namely condition (*) explained below.  The rest of this paper will be devoted to removing this condition.

The first part of this paper (\S\ref{section:Distances},\S\ref{section:HereditarySections}, and \S\ref{section:Description}) follows the proof of \cite[Theorem 4.4]{BergVanRoosmalen08} closely.  Although we reintroduce all relevant concepts, some familiarity with the proof of \cite[Theorem 4.4]{BergVanRoosmalen08} might be helpful to the reader to better understand our arguments below.

We will start our overview of the paper with \S\ref{section:HereditarySections}, where we discuss so-called \emph{split $t$-structures} (for definition, we refer to \S\ref{subsection:Splits}).  Our main result is the following theorem (compare with \cite[Theorem 1]{Ringel05}), which describes the heart of a bounded split $t$-structure.

\begin{theorem}
Let $\AA$ be an abelian category and let $\HH$ be a full subcategory of $\Db \AA$ such that $\Db \AA$ is the additive closure of $\bigcup_{t \in \bZ} \HH[t]$ and $\Hom(\HH[s],\HH[t])=0$ for $t < s$, then $\HH$ is an abelian hereditary category derived equivalent with $\AA$.
\end{theorem}

Let $\AA$ be an abelian hereditary Ext-finite category with Serre duality.  We are thus interested in finding a split $t$-structure such that the heart is of the form $\rep Q$ for a strongly locally finite thread quiver $Q$.  In particular this means that the category of projectives $\QQ$ of $\HH$ is a semi-hereditary dualizing $k$-variety, i.e. a Hom-finite Karoubian category $\QQ$ such that $\mod \QQ$ is abelian, hereditary, and has Serre duality.

To help find such $t$-structures, we introduce hereditary sections: a full additive subcategory of $\Db \AA$ is a hereditary section if there is a split $t$-structure on $\Db \AA$ and the category of projectives of its hereditary heart coincides with $\QQ$ (see Theorem \ref{theorem:SectionTilting}).

Given a hereditary section $\QQ$ in $\Db \AA$, the full replete (=closed under isomorphisms) additive subcategory generated by all indecomposables of the form $\t^n X$, $X \in \ind \QQ$ and $n \in \bZ$ will be denoted by $\bZ \QQ$.  This coincides with the full additive subcategory of $\Db \AA$ generated by all indecomposables lying in the Auslander-Reiten components of $\Db \AA$ intersecting with $\QQ$.

In order to find hereditary sections, we first introduce the (right) light cone distance $\r(-,-)$ and the round trip distance $\d(-,-)$ working on $\ind \Db \AA$ in \S\ref{section:Distances} as follows: for all $X,Y \in \ind \Db \AA$
$$\r(X,Y) = \inf \{n \in \bZ \mid \mbox{there is a path from $X$ to $\t^{-n} Y$}\}$$
and
$$\d(X,Y) = \r(X,Y) + \r(Y,X).$$
We then have following characterization of a hereditary section (Proposition \ref{proposition:SectionByDistance}): a full additive subcategory $\QQ$ of $\Db \AA$ is a hereditary section if and only if
\begin{enumerate}
\item $\r(X,Y) \geq 0$, for all $X,Y \in \ind \QQ$, and
\item if $X \in \ind \QQ$ and $\d(X,Y) < \infty$ for a $Y \in \ind \Db \AA$, then $Y \in \bZ \QQ$.
\end{enumerate}

For a set $\TT \subseteq \ind \bZ \QQ$, we define $\r(\TT,X) = \inf_{T \in \TT} \r(T,X)$, $\r(X,\TT) = \inf_{T \in \TT} \r(T,X)$, and $\d(\TT,X) = \r(\TT,X) + \r(X,\TT).$

Following the proof of \cite[Theorem 4.4]{BergVanRoosmalen08}, we find a set $\TT \in \ind \QQ$ such that $\d(\TT,X) < \infty$ for all $X \in \ind \QQ$ and we choose a hereditary section $\QQ_\TT$ such that
$$\r (\TT, X) = \left\lfloor \frac{\d (\TT, X)}{2}\right\rfloor$$
for all $X \in \ind \QQ_\TT$ where $\lfloor \cdot \rfloor$ is the floor function.

If $\TT$ is chosen to satisfy some extra properties (as given in Lemma \ref{lemma:ChooseTT}, but in particular $\TT$ has to be countable), then Theorem \ref{theorem:MainTilting} yields that the $\QQ_\TT$ is indeed a semi-hereditary dualizing $k$-variety.  Thus if $\bZ \QQ$ generates $\Db \AA$ as thick triangulated category, then $\Db \AA \cong \Db \rep Q$ for a strongly locally finite thread quiver $Q$.  That $\TT$ can indeed be chosen to satisfy the extra needed assumptions, is exactly the condition (*) mentioned earlier.

Condition (*) can easily be stated as (see \S\ref{subsection:Star}):
$$\mbox{(*) : there is a \emph{countable} subset $\TT \subseteq \ind \bZ \QQ$ such that $\d(\TT,X) < \infty$, for all $X \in \ind \bZ \QQ$.}$$

Hereditary sections not satisfying condition (*) seem to be rather artificial yet they do occur, even when the corresponding heart is, for example, generated by preprojectives (see Example \ref{example:NotStar2})!
\smallskip

We now come to the second part of the article (\S\ref{section:RoughsAndThreads} and \ref{section:GeneratedBybZQQ}) where we will remove the condition (*) from the assumptions.

The first step to understanding condition (*) better is to make a distinction between thread objects and nonthread objects in $\QQ$, whose definitions we now give.  As an easy consequence of Serre duality on $\Db \AA$, it will turn out that $\QQ$ has left and right almost split maps, thus for every $A \in \ind \QQ$, there are nonsplit maps $f:A \to M$ and $g:N \to A$ in $\QQ$ such that every nonsplit map $A \to X$ or $Y \to A$ factors though $f$ or $g$, respectively.  We will say $A$ is a thread object if both $M$ and $N$ are indecomposable.  An indecomposable object which is not a thread object will be called a nonthread object.

One major step in understanding condition (*) will be showing that there are only countably many nonthread objects (Proposition \ref{proposition:NonthreadStructure}); this will be the main result in \S\ref{subsection:Roughs}.

Thus without enlarging the set $\TT \subseteq \QQ$ above to much, we may assume it contains every nonthread object in $\ind \QQ$.  If $\bZ \QQ$ does not satisfy condition (*), then there are objects $X$ which lie ``too far from nonthread objects'', thus $\d(A,X) = \infty$ for every nonthread object $A$.  Such objects $X$ will be divided into two classes: ray objects and coray objects.  If there is a nonthread object $A$ such that $\r(A,X) < \infty$, then the thread object $X$ will be called a ray object; if there is a nonthread object $A$ such that $\r(X,A) < \infty$, then $X$ will is called a coray object.  If $\QQ$ has nonthread objects (and we may always reduce to this case), connectedness implies one of these conditions is satisfied.

On the ray objects, we define an equivalence relation (see \S\ref{subsection:Rays}) given by $X \sim Y$ if and only if $\r(X,Y) < \infty$ or $\r(Y,X) < \infty$.  A full additive category generated by an equivalence class of ray objects will be called a ray and it is shown in \ref{proposition:CountableRays} that there may only be a countable number or rays.

In order to enlarge $\QQ$, we will add an object $M$ for every ray $\RR$, called the mark of $\RR$.  This should be seen as a nonthread object ``lying on the far side of $\RR$''.

As shown in Example \ref{example:AddingMarksToD} we cannot expect to find a hereditary section $\QQ'$ such that $\QQ$ and all marks of all rays lie in $\bZ \QQ'$.  However, if $\Db \AA$ is generated by $\bZ \QQ$, then this will always be the case.

Thus in \S\ref{section:GeneratedBybZQQ} we will construct a hereditary section $\QQ'$ such that $\bZ \QQ'$ satisfies condition (*) and $\bZ \QQ \subseteq \bZ \QQ'$.  Theorem \ref{theorem:Introduction} will follow from this.

{\bf Acknowledgments} The authors like to thank Idun Reiten, Sverre Smal\o, Jan \v S\v tov\'\i\v cek, and Michel Van den Bergh for many useful discussions and helpful ideas.  The second author also gratefully acknowledges the hospitality of the Max-Planck-Institut f\"{u}r Mathematik in Bonn and the Norwegian University of Science and Technology.
\section{Conventions and Preliminaries}

\subsection{Conventions}

Throughout, let $k$ be an algebraically closed field.  All categories will be assumed to be $k$-linear.

We will fix a universe $\UU$ and assume that (unless explicitly noted) all our categories are $\UU$-categories, thus $\Hom_\CC (X,Y) \in \UU$ for any category $\CC$ and all objects $X,Y \in \Ob \CC$.  A category $\CC$ is called \emph{$\UU$-small} (or just small) if $\Ob \CC \in \UU$.

Let $\CC$ be a Krull-Schmidt category.  By $\ind \CC$ we will denote a set of chosen representatives of isomorphism classes of indecomposable objects of $\CC$.  If $\CC'$ is a Krull-Schmidt subcategory of $\CC$, we will assume $\ind \CC' \subseteq \ind \CC$.

If $\CC$ is a triangulated category with Serre duality (see below) and $\QQ$ is a full Krull-Schmidt subcategory, then we will denote by $\bZ \QQ$ the unique full additive replete (= closed under isomorphisms) subcategory of $\CC$ with $\ind \bZ \QQ = \{\tau^n X \mid X \in \ind \QQ, n \in \bZ\}$.  If $\QQ_1$ and $\QQ_2$ are Krull-Schmidt subcategories of $\CC$ such that $\bZ \QQ_1 \cong \bZ \QQ_2$ as subcategories of $\CC$, then we will say $\QQ_1$ and $\QQ_2$ are \emph{$\bZ$-equivalent}.

An \emph{(ordered) path} between indecomposables $X$ and $Y$ in a Krull-Schmidt category $\CC$ is a sequence $X=X_0, X_1, \ldots, X_n=Y$ of indecomposables such that $\Hom(X_i,X_{i+1}) \not= 0$ for all $0\leq i \leq n-1$.  A \emph{nontrivial path} is a path where there are $i,j \in \{0,1,\ldots,n\}$ such that $\rad(X_i,X_j) \not= 0$.  If there is no nontrivial path from $X$ to $X$, then we will say $X$ is \emph{directing}.

We will say a Krull-Schmidt category $\CC$ is \emph{connected} if for all indecomposables $X,Y$, there is a sequence $X=X_0, X_1, \ldots, X_n=Y$ of indecomposables such that there is either a path from $X_i$ to $X_{i+1}$ or from $X_{i+1}$ to $X_i$, for all $0\leq i \leq n-1$.

If $\CC$ is a Krull-Schmidt category and $A,B \in \ind \CC$ then we will denote by $[A,B]$ the full replete additive category containing every indecomposable $C' \in \ind C$ with $\Hom(A,C') \not= 0$ and $\Hom(C',B) \not= 0$.  We define $]A,B]$ similarly, but with the extra condition that $C' \not\cong A$.  The subcategories $[A,B[$ and $]A,B[$ are defined in an obvious way.

\subsection{Abelian hereditary categories}

An abelian category $\AA$ is said to be \emph{Ext-finite} if $\dim\Ext^i(X,Y) < \infty$ for all $i \in \bN$ and $X,Y \in \Ob \AA$.  If $\Ext^i(X,Y) = 0$ for all $i \geq 2$, then $\AA$ is called \emph{hereditary}.  If $\Ext^i(X,Y) = 0$ for all $i \geq 1$, we will say $\AA$ is \emph{semi-simple}.

For an abelian category $\AA$, we will denote by $\Db \AA$ its bounded derived category.  There is a fully faithful functor $i:\AA \to \Db \AA$ mapping every $X \in \AA$ to the complex which is $X$ in degree $0$ and $0$ in all other degrees.  We will often suppress this embedding and write $X \in \Ob \Db \AA$ instead of $iX \in \Ob \Db \AA$.

When $\AA$ is hereditary, the bounded derived category $\Db \AA$ has the following well-known description (\cite{Keller07, Lenzing07, VandenBergh01}): every object $X \in \Db \AA$ is isomorphic to the direct sums of its homologies.

\subsection{Serre duality and almost split maps}
Let $\CC$ be a $k$-linear Hom-finite triangulated category.  A \emph{Serre functor} \cite{BondalKapranov89} is a $k$-linear additive equivalence $\bS:\CC \to \CC$ such that for any two objects $A,B \in \Ob \CC$, there is an isomorphism
$$\Hom(A,B) \cong \Hom(B,\bS A)^*$$
of $k$-vector spaces, natural in $A$ and $B$.  Here, $(-)^*$ denotes the vector space dual.

A Serre functor will always be an exact equivalence.  If $\AA$ is an Ext-finite abelian category, then we will say that $\AA$ has \emph{Serre duality} if and only if $\Db \AA$ has a Serre functor.

It has been shown in \cite{ReVdB02} that an Ext-finite hereditary category has Serre duality if and only if $\AA$ has Auslander-Reiten sequences and there is a 1-1-correspondence between the indecomposable projective objects and the indecomposable injective objects via their simple top and simple socle, respectively.

It has also been shown in \cite{ReVdB02} that $\bS \cong \tau[1]$ where $\tau: \Db \AA \to \Db \AA$ is the Auslander-Reiten translate.  In particular, an Ext-finite triangulated category has Serre duality if and only if it has Auslander-Reiten triangles.

A map $f: A \to B$ is said to be \emph{left (or right) almost split} if every non-split map $A \to X$ (or $X \to B$) factors through $f$.

\subsection{Thread quivers and dualizing $k$-varieties}

We recall some definitions from \cite{Auslander74, AuslanderReiten74}.  A Hom-finite additive category $\aa$ where idempotents split will be called a \emph{finite $k$-variety}.  The functors $\aa(-,A)$ and $\aa(A,-)^*$ from $\aa$ to $\mod k$ will be called \emph{standard projective representations} and \emph{standard injective representations}, respectively.  We will write $\mod \aa$ for the category of contravariant functors $\aa \to \mod k$ which are finitely presentable by standard projectives.

Following \cite[Proposition 4.1]{BergVanRoosmalen09} we will say a finite $k$-variety $\aa$ is \emph{dualizing} \cite{AuslanderReiten74} if and only if $\aa$ has pseudokernels and pseudocokernels (thus $\mod \aa$ and $\mod \aa^\circ$ are abelian, where $\aa^\circ$ is the dual category of $\aa$), every standard projective object is cofinitely generated by standard injectives, and every standard injective object is finitely generated by standard projectives.

A finite $k$-variety $\aa$ is called \emph{semi-hereditary} if and only if the category $\mod \aa$ is abelian and hereditary.  It has been shown (\cite[Proposition 4.2]{vanRoosmalen06}, see also \cite[Theorem 1.6]{AuslanderReiten75}) that $\aa$ is semi-hereditary if and only if every full (preadditive) subcategory with finitely many elements is semi-hereditary.

Let $\aa$ be a finite $k$-variety.  It has been shown in \cite{BergVanRoosmalen09} that $\mod \aa$ is an abelian and hereditary category with Serre duality if and only if $\aa$ is a semi-hereditary dualizing (finite) $k$-varieties.  Thread quiver were then introduced in order to classify these semi-hereditary dualizing $k$-varieties.

A \emph{thread quiver} consists of the following information:
\begin{itemize}
\item A quiver $Q=(Q_0,Q_1)$ where $Q_0$ is the set of vertices and $Q_1$ is the set of arrows.
\item A decomposition $Q_1 = Q_s \coprod Q_t$.  Arrows in $Q_s$ will be called \emph{standard arrows}, while arrows in $Q_t$ will be referred to as \emph{thread arrows}.  Thread arrows will be drawn by dotted arrows.
\item With every thread arrow $\alpha$, there is an associated linearly ordered set $\TT_\alpha$, possibly empty.  When not empty, we will write this poset as a label for the thread arrow.  A finite linearly ordered poset will just be denoted by its number of elements.
\end{itemize}

When $Q$ is a thread quiver, we will denote by $Q_r$ the underlying quiver, thus forgetting labels and the difference between arrows and thread arrows.  We will say $Q$ is \emph{strongly locally finite} when $Q_r$ is strongly locally finite, i.e. all indecomposable projective and injective representations have finite dimension as $k$-vector spaces.

Let $Q$ be a strongly locally finite thread quiver.  With every thread $t \in Q_t$, we denote by $f^t: k(\cdot \to \cdot) \longrightarrow kQ_r$ the functor associated with the obvious embedding $(\cdot \to \cdot) \longrightarrow Q_r$.  We define the functor
$$f: \bigoplus_{t \in Q_t} k(\cdot \to \cdot) \longrightarrow kQ_r.$$

With every thread $t$, there is an associated linearly ordered set $\TT_t$.  We will write $\LL_t = \bN \cdot (\TT_t \stackrel{\rightarrow}{\times} \bZ) \cdot -\bN$.  Denote by
$$g^t: k(\cdot \to \cdot) \longrightarrow k\LL_t$$
a chosen fully faithful functor given by mapping the extremal points of $\cdot \to \cdot$ to the minimal and maximal objects of $\LL$, respectively.  We will write
$$g: \bigoplus_{t \in Q_t} k(\cdot \to \cdot) \longrightarrow \bigoplus_{t \in Q_t} k\LL_t.$$

The category $kQ$ is defined to be a 2-push-out of the following diagram.

$$\xymatrix{
\bigoplus_{t \in Q_t} k (\cdot \to \cdot) \ar[r]^-{f} \ar[d]_{g} & {k Q_r} \ar@{-->}^{i}[d] \\
\bigoplus_{t \in Q_t} k\LL_t \ar@{-->}[r]_-{j} & kQ
}$$

We have the following result which classifies the semi-hereditary dualizing $k$-varieties in function of strongly locally finite thread quivers.

\begin{theorem}\label{corollary:ShDualizing}\cite{BergVanRoosmalen09}
Every semi-hereditary dualizing $k$-variety is equivalent to a category of the form $kQ$ where $Q$ is a strongly locally finite thread quiver.
\end{theorem}

We will also use the following result (\cite[Corollary 6.4]{BergVanRoosmalen09}).

\begin{proposition}\label{proposition:SinksCountable}
A semi-hereditary dualizing $k$-variety has only countably many sinks and sources.
\end{proposition}

\subsection{Sketching categories}

Throughout this paper, sketches of categories (or more precisely, the Auslander-Reiten quiver) will be provided for the benefit of the reader.  All examples will be directed categories, and we will use the conventions used in \cite{vanRoosmalen06} (see also \cite{Ringel02, Ringel02b}).

We will consider only three shapes of Auslander-Reiten components: those of the form $\bZ A_\infty$, $\bZ A_\infty^\infty$, and $\bZ D_\infty$, which will be represented by squares, triangles, and triangles with a doubled side, respectively (see Figure \ref{fig:SketchesAbridged}).  These components will be ordered such that the maps go from left to right.

\begin{figure}
	\centering
		\includegraphics[width=0.40\textwidth]{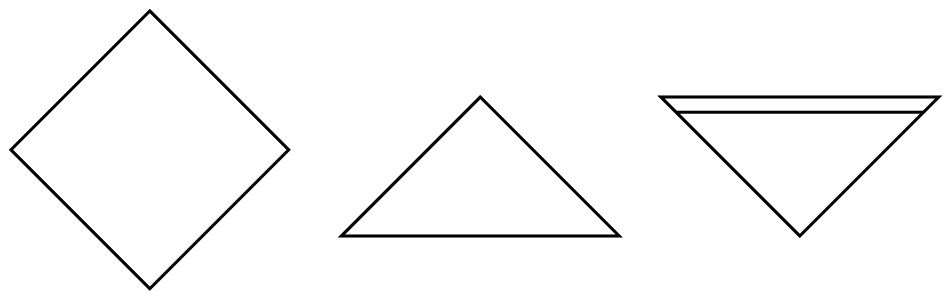}
	\caption{Sketches of $\bZ A_\infty$-, $\bZ A_\infty^\infty$-, and $\bZ D_\infty$-components}
	\label{fig:SketchesAbridged}
\end{figure}

Whenever a triangulated category comes equiped with a $t$-structure, this will be suitably indicated on the corresponding sketch.
\renewcommand{\r}{d^{\bullet}}

\section{Round Trip Distance and Light Cone Distance}\label{section:Distances}

In \cite{BergVanRoosmalen08}, the round trip distance and light cone distance were introduced for stable translation quivers of the form $\bZ Q$.  These distances proved valuable to discuss sections of $\bZ Q$.  Our goal of describing the category of projectives $\QQ$ is similar and we wish to employ similar techniques to this case.  We will have to generalize the techniques of \cite{BergVanRoosmalen08} somewhat since the category $\bZ \QQ$ does not have to be generalized standard in our present setting.  The definitions coincide in case this connecting component is generalized standard.

In this section, let $\CC = \Db \AA$ where $\AA$ is an abelian Ext-finite category with Serre duality.  Although $\AA$ is not required to be hereditary, it follows from Corollary \ref{corollary:DirectingCriterium} that our definitions are only nontrivial if $\CC$ has directing objects, which implies that $\AA$ is derived equivalent to a hereditary category (see Theorem \ref{theorem:RightClosed}).

\subsection{Light cone distance}

For all $X,Y \in \ind \CC$, we define the \emph{(right) light cone distance}\index{right light cone distance!for triangulated categories} as
$$\r(X,Y) = \inf \{n \in \bZ \mid \mbox{there is a path from $X$ to $\t^{-n} Y$}\}.$$
In particular, $\r(X,Y) \in \bZ \cup \{ \pm \infty \}$.  

\begin{remark}
Even when $X$ and $Y$ lie in the same Auslander-Reiten component, the right light cone distance does not need to coincide with the one given in \cite{BergVanRoosmalen08}, as the following example illustrates.  The difference is that the definition above takes all maps into account when determining paths, while the definition in \cite{BergVanRoosmalen08} only considers irreducible morphisms.
\end{remark}

\begin{figure}
	\exa
		\includegraphics[width=\textwidth]{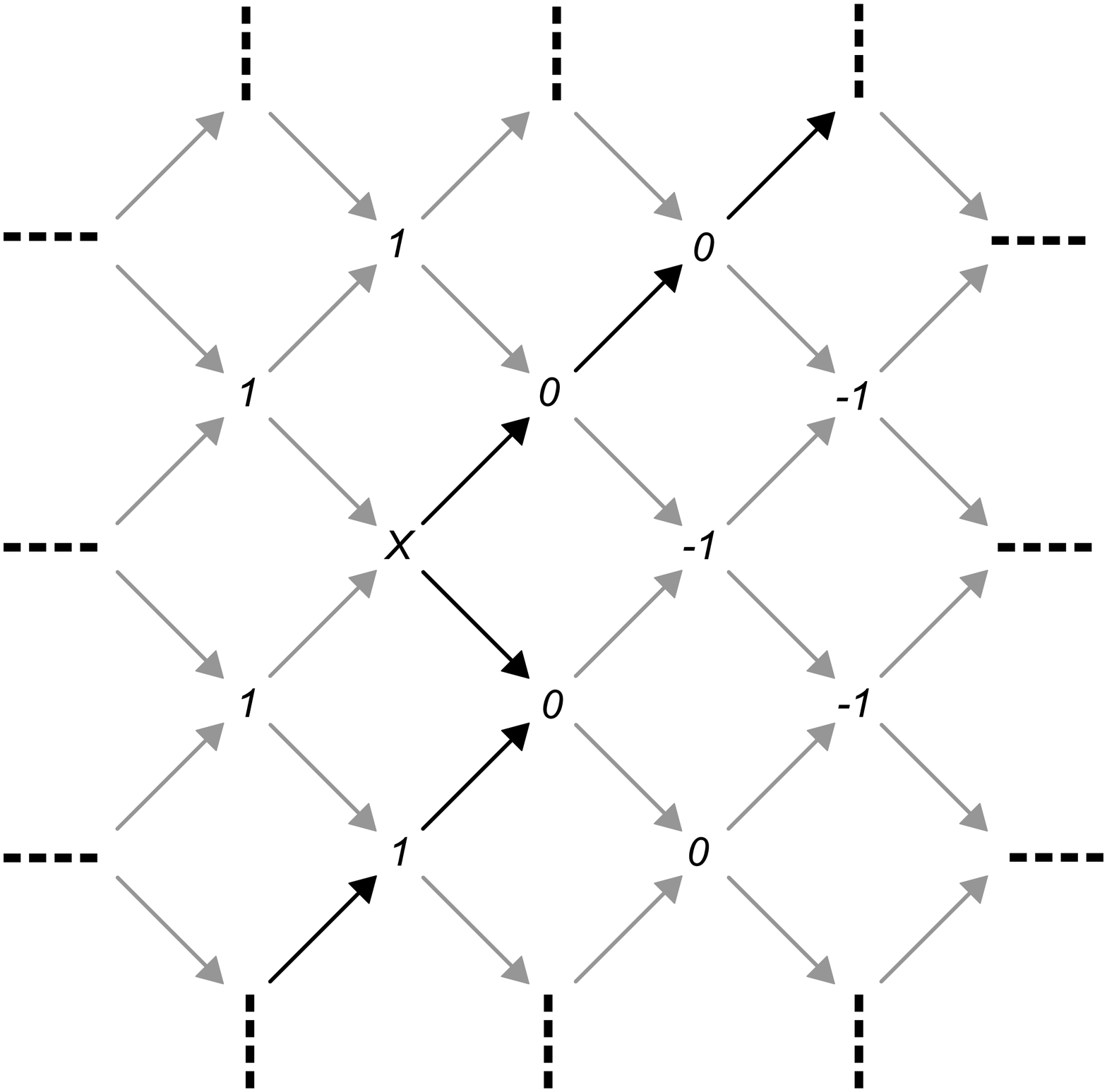}
	\exb
		\includegraphics[width=\textwidth]{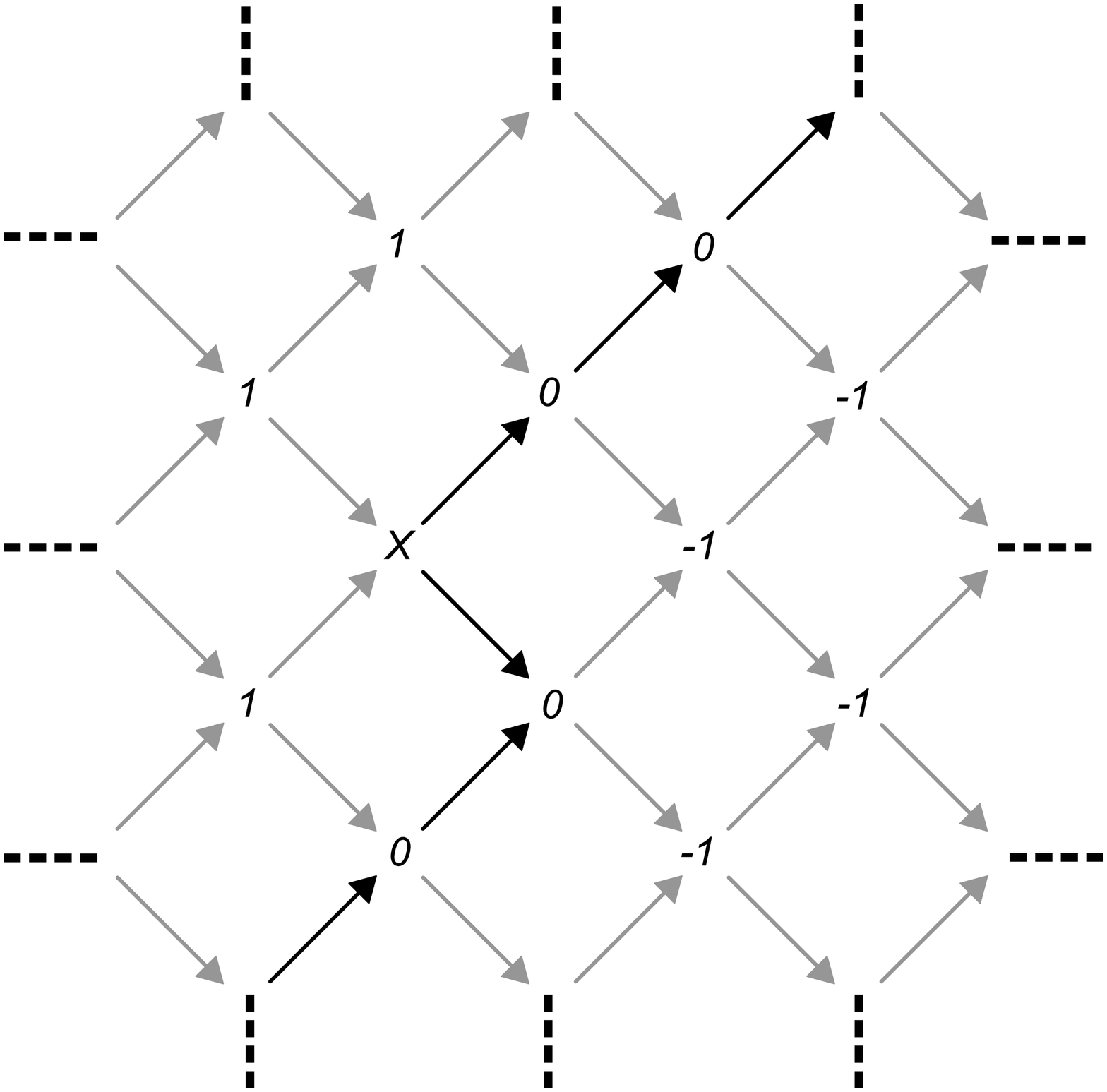}
	\exc
	\caption{The Auslander-Reiten quiver of the category $\bZ \QQ$ in Example \ref{example:DifferentDistances} where every vertex has been labeled with $\r(X,-)$.  For this, the Auslander-Reiten quiver on the left has been interpreted as a stable translation quiver, while on the right we have used the category $\bZ \QQ$ to determine the right light cone disatnce.}
	\label{fig:Thread5}
\end{figure}

\begin{example}\label{example:DifferentDistances}
Let $\aa$ be the semi-hereditary dualizing $k$-variety whose thread quiver is
$$\xymatrix@1{\bullet \ar@<3pt>@{..>}[r] \ar@<-3pt>[r] & \bullet}$$
The Auslander-Reiten quiver of $\Db \mod \aa$ containing the standard projectives of $\mod \aa$ via the standard embedding is of the form $\bZ A_\infty^\infty$.  On the left hand side of Figure \ref{fig:Thread5} we have labeled the vertices with the right light cone distance $\r(X,-)$ as a stable translation quiver (as in \cite{BergVanRoosmalen08}), while on the right hand side we have used the definition of right light cone distance given in this article.  For the benefit of the reader, the arrows between indecomposable projective objects have been drawn in black.

\end{example}

The following lemma is stated for easy reference.

\begin{lemma}\label{lemma:LightConeDistanceAndShifts}
For all $X,Y \in \ind \CC$, we have $\r(X,\t^n Y) = \r(\t^{-n}X,Y) = \r(X,Y) + n$.
\end{lemma}

Note that the function $\r$ is not symmetric.  It does however satisfy the triangle inequality.

\begin{proposition}\label{proposition:rTriangle}
For all $X,Y,Z \in \ind \CC$, we have
$$\r (X,Z) \leq \r (X,Y) + \r (Y,Z),$$
whenever this sum is defined.
\end{proposition}

\begin{proof}
The proof follows directly from the definition.
\end{proof}

For a subsets $\TT_1,\TT_2 \subseteq \ind \CC$, we define the right light cone distance in an obvious way:
$$\r(\TT_1, \TT_2) = \inf_{\begin{array}{cc} T_1 \in \TT_1 \\ T_2 \in \TT_2 \end{array}} \r(T_1,T_2).$$
The following result follows from the triangle inequality.

\begin{corollary}
Let $X \in \ind \CC$ and $\TT_1, \TT_2 \subseteq \ind \CC$, we have
$$\r (\TT_1,\TT_2) \leq \r (\TT_1,X) + \r (X,\TT_2),$$
whenever this sum is defined.
\end{corollary}

We now continue to define a \emph{right and left light cone distance sphere}\index{right light cone distance sphere!in a triangulated category}\index{left light cone distance sphere!in a triangulated category} by
$$\Sr(X,n) = \{Y \in \ind \Db \AA \mid \r(X,Y) = n\}$$
and
$$\Sl(X,n) = \{Y \in \ind \Db \AA \mid \r(Y,X) = n\},$$
respectively, for any $n \in \bZ$ and $X \in \Ob \CC$.  Finally, we will denote
$$\mbox{$\SrQQ(X,n) = \Sr(X,n) \cap \ind \QQ$ and $\SlQQ(X,n) = \Sl(X,n) \cap \ind \QQ.$}$$

\subsection{Connection with directing objects}

Although the left and right light cone distances between any two indecomposables are defined, we can only expect nontrivial results in the case where both are directing.

We start by recalling following result.

\begin{proposition}\label{proposition:RingelTriangles}\cite[Lemma 3]{Ringel05}
Let $\xymatrix@1{X \ar[r]^u &Y \ar[r]^v & Z \ar[r]^w& X[1]}$ be a triangle where $X,Y$ are indecomposable and $u$ is nonzero and noninvertible.  Let $Z_1$ be a direct summand of $Z$.  The maps $v_1:Y \to Z_1$ and $w_1:Z_1 \to X[1]$ induced by $v$ and $w$, respectively, are nonzero and noninvertible.
\end{proposition}

\begin{proposition}\label{proposition:LeftmostComposition}
Let $X, Y, Z \in \ind \CC$ such that $\r (X,Z) = 0$.  For all non-zero $f
\in \Hom(X,Y)$ and $g \in \Hom(Y,Z)$ we have that $gf$ is non-zero.  In particular, $\r(X,Z)=0$ implies $\Hom(X,Z) \not= 0$.
\end{proposition}

\begin{proof}
Without loss of generality, we may assume $g$ is not an isomorphism, and hence $C = \cone(g:Y \to Z)$ is nonzero.  It follows from Proposition \ref{proposition:RingelTriangles} that $\Hom(Z,C_i) \not= 0$ for every direct summand $C_i$ of $C$.  Using Serre duality we find $\Hom(C_i[-1],\t Z)\not=0$, and therefore $\r(C_i[-1],Z) \leq -1$.

The triangle inequality then gives $\r(X,C_i[-1]) \geq \r(X,Z) - \r(C_i[-1],Z) \geq 1$ and hence $\Hom (X,C[-1])=0$.  We deduce that $f \colon X \to Y$ does not factor through $C[-1]$ and hence $gf$ is non-zero.
\end{proof}

\begin{proposition}\label{Proposition:DirectingCriterium}
An object $X \in \ind \CC$ is directing if and only if $\r(X,X) =
0$, or equivalently, $X$ is non-directing if and only if $\r(X,X) =
-\infty$.
\end{proposition}

\begin{proof}
It is clear that directing implies $\r(X,X) = 0$.  To prove the other implication, assume there is a nontrivial path
$$X = X_0 \stackrel{f_0}{\rightarrow} X_1 \stackrel{f_1}{\rightarrow} \cdots \stackrel{f_{n-1}}{\rightarrow} X_{n} \stackrel{f_n}{\rightarrow} X.$$
Since $\r (X,X) = 0$, the triangle inequality yields $\r (X_i,X_j) = 0$ for all $i,j \in \{0, \ldots, n\}$.  Proposition \ref{proposition:LeftmostComposition} now gives that $f = f_n \ldots f_1 f_0 $ is non-zero.

Since $X$ is indecomposable, $\End X$ is a finite dimensional local algebra and thus every element is either nilpotent or invertible.  Proposition \ref{proposition:LeftmostComposition} yields $f$ is not nilpotent, hence it is invertible, a contradiction.
\end{proof}

\begin{corollary}\label{corollary:DirectingCriterium}
Let $X,Y \in \ind \CC$ such that $\r(X,Y) \in \bZ$, then both $X$ and $Y$ are directing.
\end{corollary}

\begin{proof}
Using triangle inequality, we have $\r (X,Y) \leq \r(X,X) + \r(X,Y)$, and hence $0 \leq \r(X,X)$.  We always have $\r(X,X) \leq 0$, so we get $\r(X,X) = 0$.  Proposition \ref{Proposition:DirectingCriterium} shows $X$ is directing. Showing $Y$ is directing is similar.
\end{proof}

\begin{corollary}
Let $X \in \ind \CC$.  If $X$ is directing, then so is every indecomposable $Y$ in the Auslander-Reiten component of $X$.
\end{corollary}

\begin{proof}
Since $Y$ lies in the same Auslander-Reiten component as $X$, we know $\r(X,Y) < \infty$.  Then by Proposition \ref{Proposition:DirectingCriterium} and triangle inequality, $0 = \r(X,X) \leq \r(X,Y)+\r(Y,X)$, and hence $\r(Y,X) > -\infty$.  Invoking Corollary \ref{corollary:DirectingCriterium} completes the proof.
\end{proof}

\subsection{Round trip distance}

For $X,Y \in \ind \CC$, we define the \emph{round trip distance}\index{round trip distance!for triangulated categories} $\d(X,Y)$ as the symmetrization of the right light cone distance, thus
$$\d(X,Y) = \r(X,Y) + \r(Y,X),$$
whenever this is well-defined.  It is easy to see that $\d(X,Y)$ depends only on the $\t$-orbit of $X$ and $Y$, thus $\d(X,Y) = \d(\t^n X, \t^m Y)$ for all $m,n \in \bZ$ (compare with Lemma \ref{lemma:LightConeDistanceAndShifts}).

When we restrict ourselves to indecomposables of $\bZ \QQ$, where $\QQ$ is the category of projectives of a hereditary category $\AA$ with Serre duality, then we know that both $\r(X,Y)$ and $\r(Y,X)$ will be in $\bZ \cup \{ \infty \}$, hence $\d(X,Y)$ will be well-defined.

Following proposition shows $\d$ defines a pseudometric.

\begin{proposition}\label{proposition:dDistance}
Let $\bZ \QQ$ as above.  For all $X,Y,Z \in \ind \bZ \QQ$ we have
\begin{enumerate}
\item $\d(X,Y) \geq 0$
\item $\d(X,X) = 0$
\item $\d(X,Y) = \d(Y,X)$
\item $\d(X,Z) \leq \d(X,Y) + \d(Y,Z)$
\end{enumerate}
\end{proposition}

\begin{proof}
The claims (2), (3), and (4) follow from Proposition \ref{Proposition:DirectingCriterium}, the definition, and Proposition \ref{proposition:rTriangle}, respectively.
Since then $0 = \d(X,X) \leq \d(X,Y) + \d(Y,X) = 2 \d(X,Y)$, the first claim holds as well.
\end{proof}

A \emph{round trip distance sphere}\index{round trip distance sphere!in a triangulated category} is defined in an obvious way.
\section{Hereditary sections}\label{section:HereditarySections}

Let $\AA$ be an abelian hereditary Ext-finite category with Serre duality.  In what follows, we shall discuss the category of projectives of hereditary categories $\HH$ derived equivalent to $\AA$.  These projectives will form hereditary sections in $\Db \AA$ and, likewise, a hereditary section in $\Db \AA$ will give a hereditary category $\HH$ derived equivalent to $\AA$.

We start with a some results concerning split $t$-structures. 

\subsection{Split $t$-structures}\label{subsection:Splits}

The concept of a $t$-structure was introduced by Be{\u{\i}}linson, Bernstein and Deligne in \cite{Beilinson82}.  Specifically, we will be interested in so-called split $t$-structures of which the heart will be a hereditary category (\cite{Ringel05}).

\begin{definition}\label{definition:t}
A \emph{$t$-structure} on a triangulated category $\CC$ is a pair $(D^{\geq 0}, D^{\leq 0})$ of non-zero full subcategories of $\CC$ satisfying the following conditions, where we denote $D^{\leq n} = D^{\leq 0} [-n]$ and $D^{\geq n} = D^{\geq 0} [-n]$
\begin{enumerate}
\item $D^{\leq 0} \subseteq D^{\leq 1}$ and $D^{\geq 1} \subseteq D^{\geq 0}$
\item $\Hom(D^{\leq 0}, D^{\geq 1}) = 0$
\item \label{enumerate:Triangles} $\forall Y \in \CC$, there exists a triangle $X \to Y \to Z \to X[1]$ with $X \in D^{\leq 0}$ and $Z \in D^{\geq 1}$.
\end{enumerate}
Let $D^{[n,m]} = D^{\geq n} \cap D^{\leq m}$ .  We will say the $t$-structure $(D^{\geq 0}, D^{\leq 0})$ is \emph{bounded} if and only if every object of $\CC$ is contained in some $D^{[n,m]}$.  We call $(D^{\geq 0}, D^{\leq 0})$ \emph{split} if every triangle occurring in (\ref{enumerate:Triangles}) is split.
\end{definition}

It is shown in \cite{Beilinson82} that the \emph{heart} $\HH = D^{\leq 0} \cap D^{\geq 0}$ is an abelian category.  Unfortunately, if $\AA$ is an abelian category, then not every $t$-structure on $\Db \AA$ defines a heart $\HH$ which is derived equivalent to $\AA$.  Following proposition shows that in our setting we may expect derived equivalence between $\AA$ and $\HH$.

\begin{proposition}\label{proposition:HereditaryTilting}
Let $\AA$ be an abelian category and let $(D^{\geq 0}, D^{\leq 0})$ be a bounded $t$-structure on $\Db \AA$.  If all the triangles $\tri X Y Z$ with $X \in D^{\leq 0}$ and $Z \in D^{\geq 1}$ split, then $D^{\leq 0} \cap D^{\geq 0} = \HH$ is hereditary and $\Db \AA \cong \Db \HH$ as triangulated categories.
\end{proposition}

\begin{proof}
It is well known that the category $\Ind\AA$ of left exact contravariant functors from $\AA$ to $\Mod{k}$ is a $k$-linear Grothendieck category and that the Yoneda embedding of $\AA$ into $\Ind\AA$ is a full and exact embedding.  By \cite[Proposition 2.14]{LoVdB06}, this embedding extends to a full and exact embedding $\Db \AA \to \Db \Ind \AA$.

Since all triangles $\tri X Y Z$ with $X \in D^{\leq 0}$ and $Z \in D^{\geq 1}$ split, we may use \cite[Lemma I.3.5]{ReVdB02} to see that $\HH$ is hereditary.  It is now an easy consequence of \cite[Proposition 3.1.16]{Beilinson82} that $\Db \AA \cong \Db \HH$ as triangulated categories.
\end{proof}

\begin{remark}
Since the above category $\Ind \AA$ is not a $\UU$-category, we tacitly assume an enlargement of the universe.
\end{remark}

We will say a subcategory $\DD$ of $\Db \AA$ is \emph{closed under successors} if it satisfies the following property: if $X \in \DD$ and $Y \in \ind \Db \AA$ such that $\Hom(X,Y) \not= 0$ or $Y \cong X[1]$, then $Y \in \DD$.  As the following theorem shows, this is a useful property to find split $t$-structures. 

\begin{theorem}\label{theorem:RightClosed}
Let $\AA$ be a connected abelian category and let $\DD$ be a nontrivial full subcategory of $\Db \AA$ closed under successors, then $(D^{\geq 0}, D^{\leq 0})$ is a bounded and split $t$-structure on $\Db \AA$ where $D^{\leq 0}=\DD$ and $D^{\geq 1} =  \DD^\perp$.
\end{theorem}

\begin{proof}
It is straightforward to check $(D^{\geq 0}, D^{\leq 0})$ defines a split $t$-structure.  It follows from \cite[Lemma 7]{Ringel05} that it is also bounded.
\end{proof}

Combining the previous theorem with \cite[Theorem 1]{Ringel05}, we get the following attractive description of a hereditary heart in a derived category.

\begin{corollary}\label{corollary:InverseOfHereditary}
Let $\AA$ be an abelian category and let $\HH$ be a full subcategory of $\Db \AA$ such that $\Db \AA$ is the additive closure of $\bigcup_{t \in \bZ} \HH[t]$ and $\Hom(\HH[s],\HH[t])=0$ for $t < s$, then $\HH$ is an abelian hereditary category derived equivalent with $\AA$.
\end{corollary}

\begin{remark}
In \cite[Theorem 1]{Ringel05} one starts with a full subcategory $\HH$ of \emph{a} triangulated category $\TT$ (not necessarily a derived category) and obtains that $\TT$ is equivalent \emph{as additive category} to $\Db \HH$, for a hereditary category $\HH$.  In Theorem \ref{theorem:RightClosed} and Corollary \ref{corollary:InverseOfHereditary} we restrict ourselves to the case where $\TT$ is a derived category (with the induced triangulated structure) and find $\TT \cong \Db \HH$ \emph{as triangulated categories}.
\end{remark}

\subsection{Definition and characterization of hereditary sections}

Before defining a hereditary section, we need a preliminary concept.  Throughout, let $\AA$ be an abelian Ext-finite category with Serre duality and write $\CC = \Db \AA$.

\begin{definition}
Let $\QQ$ be a full subcategory of $\CC$.  We will say $\QQ$ is \emph{convex} if every path in $\CC$ starting and ending in $\QQ$ lies entirely in $\QQ$. A subcategory $\QQ$ of $\CC$ is called \emph{$\t$-convex} if $\bZ \QQ$ is convex.
\end{definition}

\begin{example}
Any object $X \in \ind \CC$ spans a convex subcategory $\QQ$ of $\CC$ if and only if $X$ is directing in $\CC$.
\end{example}

\begin{remark}
Since there is always a trivial path between isomorphic objects, a convex subcategory will always be replete.
\end{remark}

In what follows $\QQ$ will consists only of directing objects.  In this case, we may give an alternative formulation of $\tau$-convex: $\QQ$ will be $\tau$-convex if and only if for every $X \in \ind \CC$, the condition $\d(\QQ,X) \not= \infty$ implies that $\QQ$ meets the $\tau$-orbit of $X$.

\begin{definition}
A \emph{hereditary section} is a nontrivial (= having at least one nonzero object), full, convex, and $\t$-convex additive subcategory $\QQ$ of $\CC$ such that $\QQ$ meets every $\tau$-orbit at most once.
\end{definition}

\begin{remark}
The notion of a hereditary section is self-dual.  If $\QQ$ is a hereditary section in $\CC$, then $\QQ^\circ$ is a hereditary section in $\CC^\circ$.
\end{remark}

\begin{remark}
If $\AA$ is semi-simple, then $\bS \cong id_{\Db \AA}$ such that $\tau \cong [-1]$.  Since a hereditary section $\QQ$ of $\Db \AA$ may meet every $\tau$-orbit at most once, we have that $X \in \Ob \QQ$ implies that $X[n] \not\in \Ob \QQ$ for all $n \not= 0$.
\end{remark}

\begin{example}
If $\AA$ is a hereditary abelian Ext-finite category with Serre duality with $\QQ_\AA$ as category of projectives, then $\QQ_\AA$ is a hereditary section in $\Db \AA$.  In Theorem \ref{theorem:SectionTilting} the converse of this statement will be shown.
\end{example}

\begin{proposition}\label{proposition:SectionByDistance}
The subcategory $\QQ$ is a hereditary section if and only if it is a full and $\t$-convex additive subcategory
$\QQ$ of $\CC$ such that $\r (X, Y) \geq 0$ for all $X, Y \in \ind \QQ$.
\end{proposition}

\begin{proof}
We may assume $\Db \AA$ is connected.  Furthermore, the statement is trivial if $\AA$ is semi-simple, thus assume the global dimension of $\AA$ is at least one.

Assume $\QQ$ is a hereditary section in $\Db \AA$.  If $\r (X, Y) < 0$, then there is a path from $X$ to $\tau Y$.  Since $\AA$ is not semi-simple, there is also a path from $\tau Y$ to $Y$ and thus, using that $\QQ$ is convex, we see that $\tau Y \in \QQ$, a contradiction.  This proves one direction.

Assume $\QQ$ is a full and $\t$-convex additive subcategory of $\ind \CC$ such that $\r (X, Y) \geq 0$ for all $X,
Y \in \QQ$.  Since $\r (X,  \t^{-n} X) < 0$ for all $n > 0$,
$\QQ$ contains at most one object from each $\t$-orbit.

Assume $X, Y \in \QQ$ with paths from $X$ to $Z$ and from $Z$ to $Y$, thus $\r(X,Z) \leq 0$ and $\r(Z,Y) \leq 0$. Since $\QQ$ is $\t$-convex, $\QQ$ contains an object of the $\t$-orbit of $Z$.  Using the triangle inequality, we find $\r(X,Y) \leq \r(X,Z) + \r(Z,Y) \leq 0$.  Since we have assumed $\r(X,Y) \geq 0$, we see $\r (X, Z) = 0$ and $\r (Z, Y) = 0$.  Thus Lemma \ref{lemma:LightConeDistanceAndShifts} shows that the object $\QQ$ contains from the $\tau$-orbit of $Z$ must be $Z$ itself.  Hence $\QQ$ is convex.
\end{proof}

We now come to the main result about hereditary sections, characterizing them to be categories of projectives of a hereditary heart.

\begin{theorem}\label{theorem:SectionTilting}
Let $\AA$ be a connected Ext-finite abelian category with Serre duality and let $\QQ$ be a hereditary section of $\Db \AA$, then there exists an Ext-finite abelian hereditary category $\HH$ with Serre duality, such that $\AA$ is derived equivalent to $\HH$ and the category of projectives of $\HH$ is given by $\QQ$.
\end{theorem}

\begin{proof}
If $\AA$ is semi-simple, the category $\HH$ is just $\QQ$ itself.  Thus assume now that $\AA$ is not semi-simple.

Let $\DD$ be the full replete additive subcategory of $\Db \AA$ spanned by all indecomposable objects $X$ with $\r(X,\QQ) \geq 0$ and $\r(\QQ,X) < \infty$.  We check that $\DD$ satisfies the conditions of Theorem \ref{theorem:RightClosed}.

Let $X \in \ind \QQ$.  Since $\r(X,\QQ) = \r(\QQ,X) = 0$ we know that $X \in \Ob \DD$, and Lemma \ref{lemma:LightConeDistanceAndShifts} shows that $\t X \not \in \Ob \DD$ such that $\DD$ is indeed a nontrivial subcategory of $\Db \AA$.

Let $X \in \ind \DD$ and $Y \in \ind \Db \AA$ such that $\Hom(X,Y) \not= 0$, or thus in particular $\r(X,Y) \leq 0$.  The triangle inequality implies that $Y \in \ind \DD$.  Furthermore, the conditions on $\AA$ imply there is an Auslander-Reiten triangle $X \to M_X \to \t^- X \to X[1]$, such that Proposition \ref{proposition:RingelTriangles} yields that $\r(X,X[1]) \leq 0$.  As above the triangle inequality will implies that $X[1] \in \DD$.

We conclude that the conditions of Theorem \ref{theorem:RightClosed} are indeed satisfied such that there is a split $t$-structure on $\Db \AA$ with $\DD^{\leq 0} = \DD$ and $\DD^{\geq 1} = \DD^\perp$.  Denote the hereditary heart by $\HH$.  We only need to show that the catgeory of projectives $\QQ_\HH$ of $\HH$ coincides with $\QQ$ in $\Db \AA$.

Note that $\QQ \subseteq \DD$ and $\QQ[-1] \subseteq \DD^\perp$, so that $\QQ \subseteq \HH$.

Let $X \in \ind \QQ_\HH$, thus $X \in \ind \DD$ but $\t X \not\in \ind \DD$.  In this case, we have $\r(X,\QQ) = 0$ such that $\tau$-convexity implies that $X \in \ind \QQ$.  If $X \in \ind \QQ$, then $\t X \not\in \ind \DD$ but $X \in \HH$ such that $X$ is a projective object in $\HH$ and hence $X \in \ind \QQ_\HH$.  We conclude that $\QQ \cong \QQ_\HH$ as subcategories of $\Db \AA$.
\end{proof}

\begin{observation}\label{observation:PropertiesSection}
Since every hereditary section is the image of the category of projectives of a hereditary category in its derived category, we see that every hereditary section $\QQ$ of $\CC$ is semi-hereditary, a partial tilting set, has left and right almost split maps, and consists of only directing objects.
\end{observation}

\begin{remark}
Theorem \ref{theorem:SectionTilting} shows that, given a hereditary section $\QQ$, there is a $t$-structure on $\Db \AA$ such that $\QQ$ is the category of projectives of the heart $\HH$.  However, the $t$-structure is not uniquely determined by $\QQ$ as the next example illustrates.
\end{remark}

\begin{example}\label{example:NotUniqueTilt}
Let $Q$ be the thread quiver $\xymatrix@1{x \ar@{..>}[r]^{1} & y}$, and let $\QQ$ be the standard hereditary section in $\Db \rep Q$.  Denote by $P_x \in \QQ$ the indecomposable object associated with $x$.  The category $\Db \rep Q$ is sketched in Figure \ref{fig:NotUniqueTilt}.

We find a smaller hereditary section $\QQ'$ spanned by all objects $A \in \QQ$ with $\d(P_x,A) < \infty$, thus $\QQ'$ contains exactly those indecomposables of $\QQ$ which lie in the same Auslander-Reiten component of $\Db \rep Q$ as $P_x$.

There are at least two different hearts in $\Db \rep Q$ such that $\QQ'$ is the category of projectives, as shown in Figure \ref{fig:NotUniqueTilt}.  The middle picture corresponds to the $t$-structure given in the proof of Theorem \ref{theorem:SectionTilting}.
\begin{figure}[tb]
	\centering
		\includegraphics[width=0.60\textwidth]{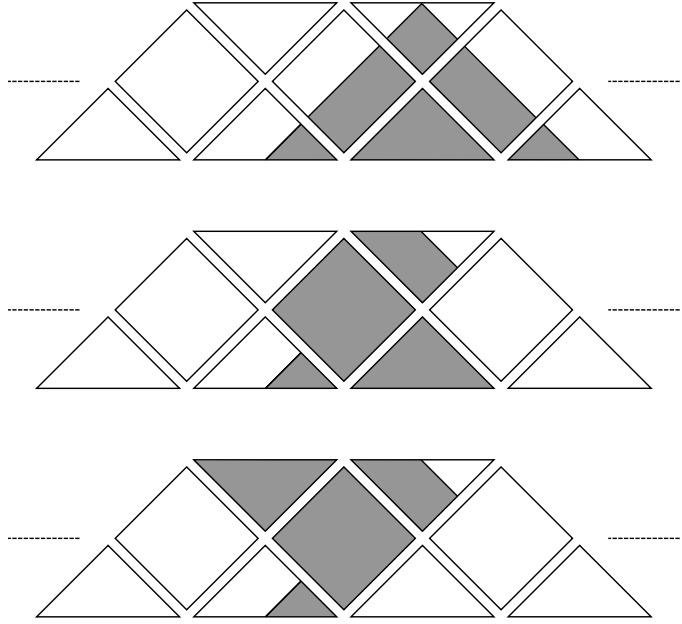}
	\caption{Illustration of Example \ref{example:NotUniqueTilt}}
	\label{fig:NotUniqueTilt}
\end{figure}
\end{example}

The following statement is a special case of Proposition \ref{proposition:SectionByDistance}.

\begin{corollary}\label{corollary:TauTilting}
Let $\AA$ be an abelian hereditary Ext-finite $k$-linear category satisfying Serre duality and let $\QQ$ be the category of projectives of $\AA$.  Let $\QQ'$ be a full preadditive subcategory of $\Db \AA$ such that $\bZ \QQ = \bZ \QQ'$ and $\r (X, Y) \geq 0$ for all $X,Y \in \ind \QQ'$, then $\QQ'$ is a hereditary section in $\Db \AA$.
\end{corollary}

\subsection{(Co)reflective subcategories of hereditary sections}

We prove an analogue of Proposition \cite[Proposition 5.2]{BergVanRoosmalen09} for hereditary sections.  We introduce the following notation.  Let $\AA$ be an Ext-finite abelian category and $\QQ$ a hereditary section in $\Db \AA$.  For any object $\ZZ \subseteq \Ob \Db \AA$, we define $\QQ_{^\perp \ZZ}$ as the full subcategory of $\QQ$ left-orthogonal on $\ZZ$, thus
$$A \in \Ob \QQ_{^\perp \ZZ} \Leftrightarrow \forall Z \in \ZZ: \RHom(A,Z) = 0.$$

\begin{proposition}\label{proposition:WhenAdjoints}
Let $\QQ$ be a hereditary section in $\Db \AA$.
\begin{enumerate}
\item Let $\ZZ \subset \Ob \Db \AA$ with $\sum_{Z \in \ZZ} \dim \Hom(A,Z) < \infty$ for all $A \in \QQ$ and $\Hom(Z_1,Z_2[n]) = 0$ for all $Z_1, Z_2 \in \Ob \Db \AA$ and $n \in \bZ\setminus\{0\}$.  Then the embedding $\QQ_{^\perp \ZZ} \to \QQ$ has a left and a right adjoint.
\item Let $X,Y \in \QQ$.  The embedding $[X,Y] \to \QQ$ has a left and a right adjoint.
\end{enumerate}
\end{proposition}

\begin{proof}
Theorem \ref{theorem:SectionTilting} yields there is a hereditary category $\HH \subset \Db \AA$ with $\QQ$ as its category of projectives.  The proof is obtained by repeating the proof of \cite[Proposition 5.2]{BergVanRoosmalen09} in $\HH$.
\end{proof}

\subsection{Criterium for being a dualizing $k$-variety}

We will be interested in hereditary sections which are dualizing $k$-varieties.  The following criterion will be useful.

\begin{proposition}\label{proposition:ConditionsToBeDualizing}
A hereditary section $\QQ$ is dualizing if and only if for every $A \in \ind \QQ$ there are $C_1, C_2 \in \Ob \QQ$ such that for every $B \in \ind \QQ$
\begin{enumerate}
\item $\Hom(B,C_1) \not= 0$ when $\r(A,B) = 0$, and
\item $\Hom(C_2,B) \not= 0$ when $\r(B,A) = 0$.
\end{enumerate}
\end{proposition}

\begin{proof}
Let $\HH$ be a hereditary category of which $\QQ$ is the category of projectives (Theorem \ref{theorem:SectionTilting}).  The first statement is equivalent to saying there is an epimorphism $\QQ(-,C_1) \to \QQ(A,-)^*$ and the second statement is equivalent to saying there is a monomorphism $\QQ(-,A) \to \QQ(C_2,-)^*$.  Since the cokernel of the first map and the kernel of the second map are finitely generated projectives, we know that $\QQ(-,A)$ is cofinitely presented and $\QQ(A,-)$ is finitely presented.

By Observation \ref{observation:PropertiesSection} $\QQ$ is semi-hereditary and thus Corollary \ref{corollary:ShDualizing} yields the required result.
\end{proof}
\subsection{Light cone}

Let $\AA$ be an abelian category with Serre duality and $X \in \Db \AA$ be an indecomposable directing object.  We define the light cone centered on $X$ to be full replete additive category $\QQ_X$ with $\ind\QQ_X = \Sr (X,0)$, thus $\QQ_X$ is generated by those indecomposable objects $Y$ such that $X$ admits a path to $Y$, but no path to $\t Y$.  Using Proposition \ref{proposition:SectionByDistance} one easily checks that $\QQ_X$ is a hereditary section.

If $\AA$ is connected then  Theorem \ref{theorem:SectionTilting} shows that $\QQ_X$ defines a $t$-structure with heart a hereditary category $\HH_X$.  We will refer to $\HH_X$ as the light cone tilt centered on $X$.  A similar construction has been used by Ringel in \cite{Ringel05}.

Dually we define the co-light cone and the co-light cone tilt centered on $X$.

\begin{lemma}
In the light cone tilt centered on $X$, we have $\Hom(X,P) \not= 0$, for all projectives $P$.
\end{lemma}

\begin{proof}
The result follows directly from Proposition \ref{proposition:LeftmostComposition}.
\end{proof}

\begin{lemma}\label{lemma:ProjectivesCopresented}
In the light cone tilt centered on $X$, all projectives objects have an injective resolution.
\end{lemma}

\begin{proof}
Let $P$ be a projective and consider the canonical map $P \to \bS X \otimes \Hom(P,\bS X)^*$ with kernel $K$. Since $P$ is projective, the kernel needs to be projective as well.

It is straightforward to check that $\Hom(X, K) = 0$, hence $K = 0$ and the canonical map is a monomorphism.  An injective resolution is then given by $0 \to P \stackrel{f}{\rightarrow} \bS X \otimes \Hom(P,\bS X)^* \to \coker f \to 0$.
\end{proof}

\begin{proposition}\label{proposition:PreprojectivesCopresented}
In a light cone tilt, all preprojectives have projective and injective resolutions.
\end{proposition}

\begin{proof}
It suffices to show this for all indecomposable preprojective objects.  Every such object is of the form $\t^{-n} Y$ for an indecomposable projective object $Y$.  We will prove the statement by induction on $n$.  If $n=0$ then the statement is Lemma \ref{lemma:ProjectivesCopresented}.

Assume that $\t^{-n} Y$ has a projective and an injective resolution.  If $0 \to \t^{-n} Y \to I \to J \to 0$ is an injective resolution of $\t^{-n} Y$ then $0 \to \bS^{-1} I \to \bS^{-1}J \to \t^{-n-1} Y \to 0$ is a projective resolution of $\t^{-n-1} Y$.  Since the projectives $\bS^{-1}I$ and $\bS^{-1}J$ have injective resolutions, the same holds for $\t^{-n-1} Y$.
\end{proof}

\section{Hereditary sections $\bZ$-equivalent to dualizing $k$-varieties}\label{section:Description}

\subsection{The condition (*)} \label{subsection:Star}

Let $\AA$ be a connected abelian hereditary Ext-finite category satisfying Serre duality and denote the category of projectives by $\QQ$.  We will assume $\bZ \QQ$ is connected.

If $\QQ$ is a dualizing $k$-variety, then $\QQ(-,A)$ is cofinitely presented.  This means that at least one source $S$ maps non-zero to $A$, hence $\r(S,A)=0$.  Dually we find that $A$ maps non-zero to at least one sink $T$, such that $\r(A,T)=0$.

Proposition \ref{proposition:SinksCountable} yields there are only a countable amount of sinks and sources, hence $\QQ$ satisfies the following property: there is a countable subset $\TT \subseteq \ind \QQ$ such that $\d(\TT,X)=0$, for all $X \in \ind \QQ$.

We will weaken this property to :
$$\mbox{(*) : there is a countable subset $\TT \subseteq \ind \bZ \QQ$ such that $\d(\TT,X) < \infty$, for all $X \in \ind \bZ \QQ$.}$$

It is thus clear (*) needs to be satisfied when $\QQ$ is a dualizing $k$-variety.  Moreover if there is a hereditary section $\QQ'$ in $\Db \AA$ with $\bZ \QQ = \bZ \QQ'$ where $\QQ'$ is a dualizing $k$-variety, then $\bZ \QQ$ also needs to satisfy condition (*).

Before giving an example where condition (*) is not satisfied, we recall following definitions.

\begin{definition}
Let $\PP$ be a poset.  The subset $\TT \subseteq \PP$ is said to be \emph{cofinal} if for every $X \in \PP$ there is a $Y \in \TT$ such that $X \leq Y$.  The least cardinality of the cofinal subsets of $\PP$ is called the \emph{cofinality} of $\PP$ and is denote by $\cofin \PP$.

Dually, one defines a \emph{coinitial} subset of $\PP$ and the \emph{coinitiality} of $\PP$ is denoted by $\coinit \PP$.
\end{definition}

Next example shows (*) is not always satisfied.

\begin{example}\label{example:NotStar}
Let $\LL$ be a linearly ordered and locally discrete set such that $\cofin \LL > \aleph_0$.  For example, if $\TT$ is a linearly ordered set with $\cofin \TT > \aleph_0$ we may define the poset $\LL = \TT \stackrel{\rightarrow}{\times} \bZ$.

Let $\PP$ be the poset $\bN \cdot (\TT \stackrel{\rightarrow}{\times} \bZ) \cdot -\bN$, thus $k\PP$ is the semi-hereditary dualizing $k$-variety given by the thread quiver
$\xymatrix@1@C+10pt{\cdot \ar@{..>}[r]^{\TT} & \cdot}$.
We may sketch the category as the upper part of Figure \ref{fig:NotStar}.

\begin{figure}[tb]
	\centering
	\psfrag{L}[][]{$\LL$}
		\includegraphics[width=0.60\textwidth]{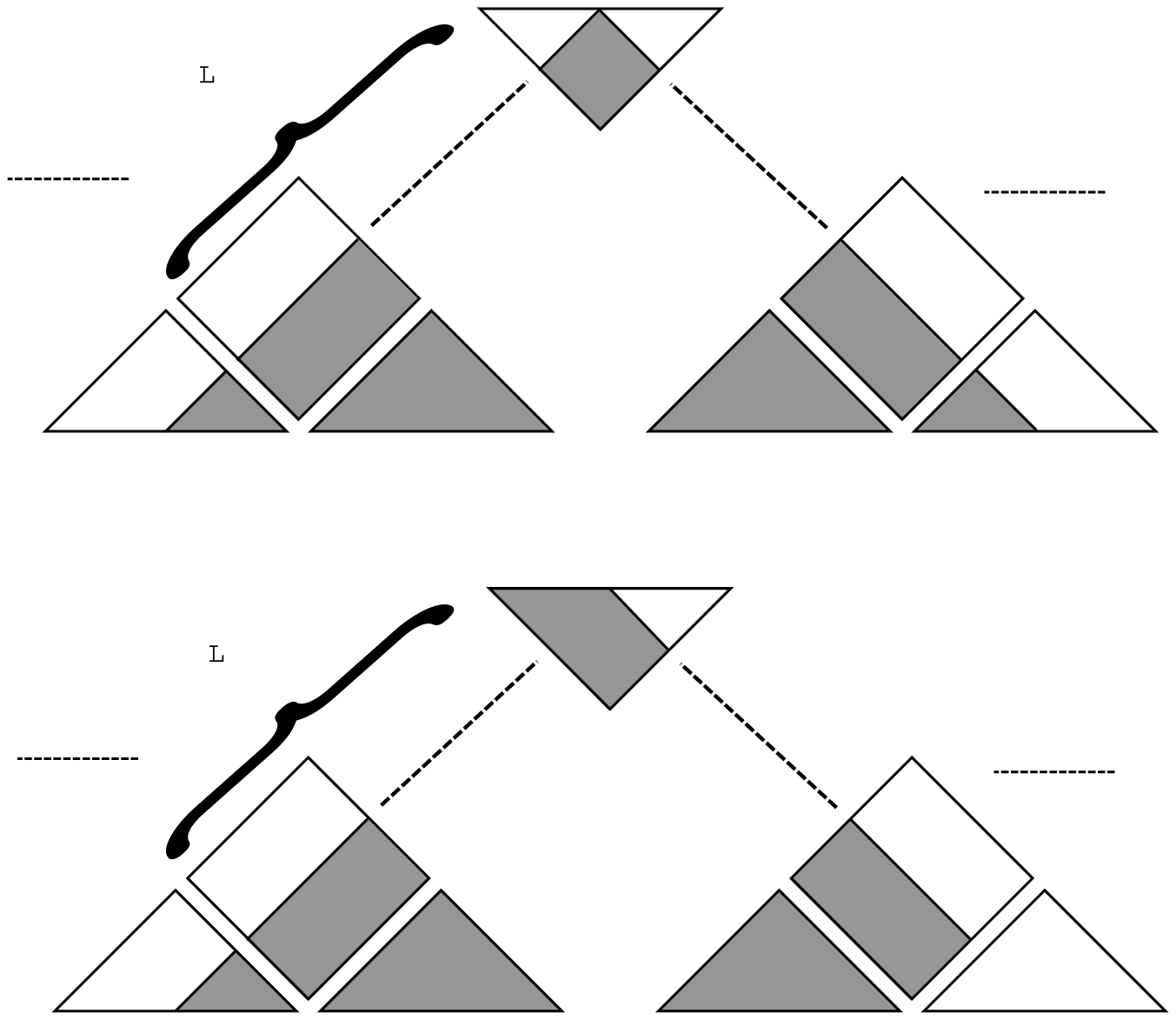}
	\caption{Illustration of Example \ref{example:NotStar}}
	\label{fig:NotStar}
\end{figure}

In $\mod k\PP$, we consider a new hereditary category $\HH$ by choosing a hereditary section $\QQ$ in $\mod k\PP$ generated by all standard projectives of the form $\PP(-,A)$ where $A \in \bN$ or $A \in \LL$.  The category $\HH$ is marked with gray in Figure \ref{fig:NotStar}.

The new category $\HH$ has category of projectives $\QQ$ and $\bZ \QQ$ does not satisfy (*).
\end{example}

\begin{example}\label{example:NotStar2}
Let $\HH'$ be the dual category of the category $\HH$ defined in Example \ref{example:NotStar} (see Figure \ref{fig:NotStar2}).  This category is generated by preprojective objects.  Denote by $\QQ'$ the category of projectives of $\HH'$.  It is clear that $\bZ \QQ'$ does not satisfy condition (*).
\begin{figure}[tb]
	\centering
	\psfrag{L}[][]{$\LL$}
		\includegraphics[width=0.60\textwidth]{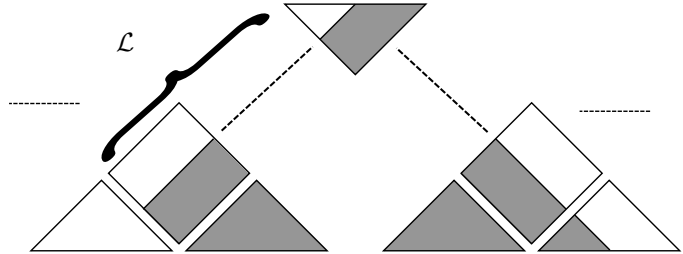}
	\caption{Sketch of a category generated by preprojective objects, but which does not satisfy condition (*).}
	\label{fig:NotStar2}
\end{figure}
\end{example}

The following lemma says that, under the condition (*), we can choose the set $\TT$ to satisfy some additional properties.

\begin{lemma}\label{lemma:ChooseTT}
Let $\QQ$ be a hereditary section such that $\bZ \QQ$ satisfy condition (*).  There is a countable subset $\TT=\{T_i\}_{i \in I} \subseteq \ind \bZ \QQ$, with $I \subseteq \bN$, satisfying the following properties.
\begin{enumerate}
\item $\d(\TT,X) < \infty$ for all $X \in \ind \bZ \QQ$,
\item $\d(\TT_j,T_k) = \infty$ for all $j < k$ and where $\TT_j=\{T_i\}_{i \leq j}$,
\item $\r(T_i,T_j) \geq \max\{i,j\}$ for all $i\not=j$.
\end{enumerate}
\end{lemma}

\begin{proof}
The first condition is exactly condition (*), so we may assume there is a countable subset $\TT=\{T_i\}_{i \in I} \subseteq \ind \bZ \QQ$ satisfying the first property.

For the second property, consider $\TT' =\{T_k \in \TT \mid \forall j<k:\d(\TT_j, T_k) = \infty\}$ instead of $\TT$.   It is clear that $\TT' \subseteq \ind \bZ \QQ$ satisfies the second condition.  It follows from the triangle inequality that $\d(\TT',X) < \infty$ for all $X \in \bZ \QQ$.

For the last property, assume $\TT=\{T_i\}_{i \in I} \subseteq \ind \bZ \QQ$ is a countable set satisfying the first two properties.  To ease notations, assume $I=\{0,1, \ldots, n\}$ or $I=\bN$.  We will define sets $\SS_i$ recursively.  Firstly let $\SS_0 = \{T_0\}$.  For every $i > 0$, choose an object $S_i$ on the $\t$-orbit of $T_i$ such that $\r(\SS_{i-1},S_i) \geq i$ and $\r(S_i,\SS_{i-1}) \geq i$ (see Lemma \ref{lemma:LightConeDistanceAndShifts}).  This is possible since, by the second condition, one of these will be infinite.

The set $\SS = \cup_{i \in I} \SS_i$ satisfies the required properties.
\end{proof}
\subsection{Finding a dualizing $k$-variety $\bZ$-equivalent to $\QQ$}

Let $\AA$ be a connected Ext-finite abelian category with Serre duality and let $\QQ$ be a hereditary section.  We have remarked above that $\QQ$ (or $\bZ \QQ$) needs to satisfy condition (*) for there to be a hereditary section $\QQ'$ which is a dualizing $k$-variety and $\bZ$-equivalent to $\QQ$.  The main result of this section will be to show the condition (*) is also sufficient, namely if $\QQ$ is a hereditary section in $\Db \AA$ such that there is a countable set $\TT \subseteq \ind \bZ \QQ$ with $\d(\TT,X) < \infty$ for all $X \in \ind \bZ \QQ$, then $\QQ$ is $\bZ$-equivalent to a semi-hereditary dualizing $k$-variety $\QQ_\TT$.

We start by choosing such a set $\TT$ and constructing an associated hereditary section $\QQ_\TT$.  We will then show that $\QQ_\TT$ is a dualizing $k$-variety.

\begin{construction}\label{construction:ChoosingTilt}
We start by choosing a set $\TT$ with the properties of Lemma \ref{lemma:ChooseTT}.  Associated to this set $\TT$, we will consider the full subcategory $\QQ_\TT$ of $\Db \AA$ as follows: for every $X \in \ind \bZ \QQ$, fix a $\tau$-shift of $X$ such that
$$\r (\TT, X) = \left\lfloor \frac{\d (\TT, X)}{2}\right\rfloor.$$
\end{construction}

\begin{example}\label{example:LargerTilt}
Let $\aa$ be the dualizing $k$-variety given by the thread quiver $\xymatrix@1{\cdot\ar@{..>}[r]^{2}&\cdot}$ thus $\aa$ is equivalent to $k(\bN \cdot \bZ \cdot \bZ \cdot -\bN)$.  The Auslander-Reiten quiver of $\Db \mod \aa$ is as sketched in the upper part of Figure \ref{fig:LargerTiltSketch}.  We will consider the hereditary section $\QQ$ spanned by all objects of $\aa \subset \Db \mod \aa$ lying in a $\bZ A_\infty^\infty$-component.  The corresponding hereditary category $\AA$ is as given by the middle part of Figure \ref{fig:LargerTiltSketch}.

\begin{figure}[tb]
	\centering
		\includegraphics[width=0.50\textwidth]{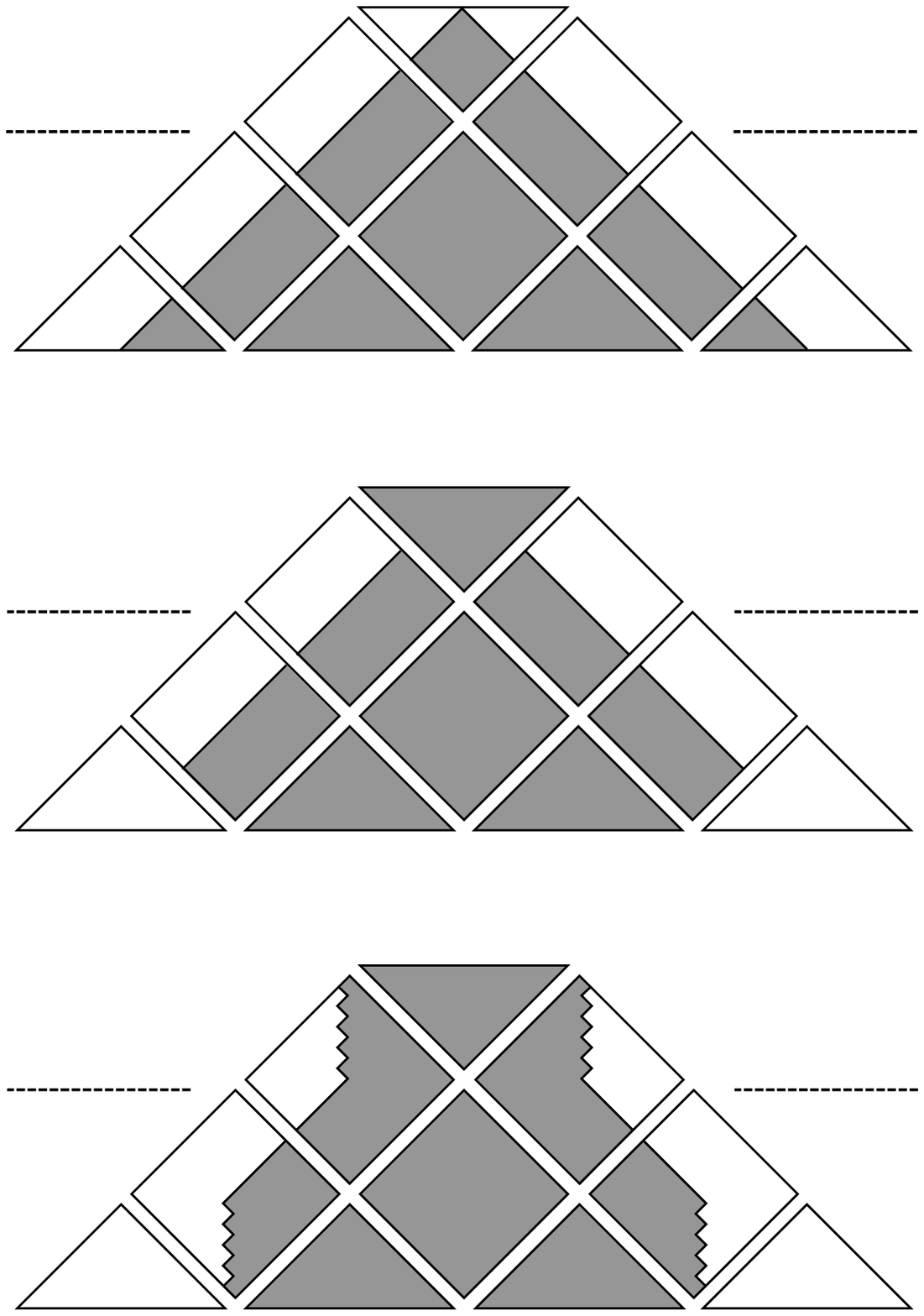}
	\caption{Illustrations for $\mod \aa$, $\AA$, and $\HH$ of Example \ref{example:LargerTilt}}
	\label{fig:LargerTiltSketch}
\end{figure}

We choose a set $\TT = \{T_0, T_1\}$ as in Figure \ref{fig:LargerTilt}, satisfying the conditions $\d(T_0,T_1) = \infty$ and $\r(T_0,T_1) \geq 1$ from Lemma \ref{lemma:ChooseTT}.  In Figure \ref{fig:LargerTilt}, the light cones $\Sr(\TT,0)$ and $\Sl(\TT,0)$ have been marked by black arrows, and the corresponding full subcategory $\QQ_\TT$ of $\Db \AA$ has been indicated by '$\bullet$'.
\begin{figure}[tbp]
	\centering
	  \psfrag{x}[][]{$T_0$}
		\psfrag{Y}[][]{$T_1$}
		\psfrag{B}[][]{$\bullet$}
		\includegraphics[width=0.90\textwidth]{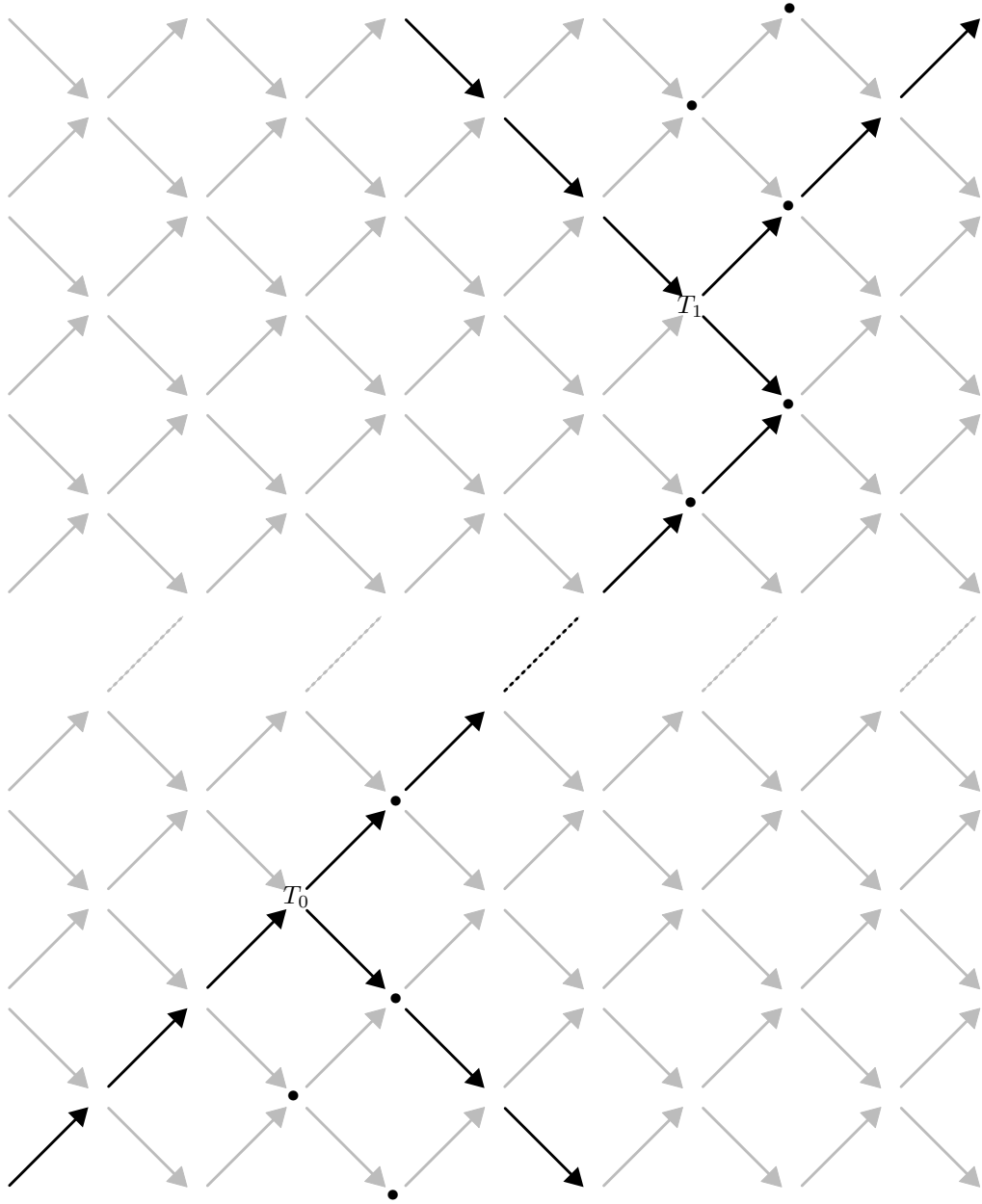}
	\caption{The light cones and chosen hereditary section of Example \ref{example:LargerTilt}}
	\label{fig:LargerTilt}
\end{figure}
\end{example}

We first verify that $\QQ_\HH$ defined above is indeed a hereditary section.

\begin{proposition}\label{proposition:IsHereditarySection}
The subcategory $\QQ$ defined in Construction \ref{construction:ChoosingTilt} is a hereditary section.
\end{proposition}

\begin{proof}
According to Corollary \ref{corollary:TauTilting} we only need to check that $\r (Y, Z) \geq 0$ for all $Y, Z \in \ind \QQ_\HH$.  Using the triangle inequality, we find
\begin{eqnarray*}
\r (Y, Z) &\geq& \r(\TT,Z) - \r(\TT,Y) \\
&=& \left\lfloor \frac{\d (\TT, Z)}{2}\right\rfloor - \left\lfloor \frac{\d (\TT, Y)}{2}\right\rfloor \geq 0
\end{eqnarray*}
if $\d (\TT, Y) \leq \d (\TT, Z)$, and
\begin{eqnarray*}
\r (Y,Z) &\geq& \r(Y,\TT) - \r(Z,\TT) \\
&=& \left\lceil \frac{\d (\TT, Y)}{2}\right\rceil - \left\lceil \frac{\d (\TT, Z)}{2}\right\rceil \geq 0.
\end{eqnarray*}
if $\d (\TT, Z) \leq \d (\TT, Y)$.
\end{proof}

\begin{lemma}\label{lemma:HomsInQ}
Let $A, B \in \ind \QQ_\TT$ with $\Hom (A, B) \not= 0$, then
\begin{enumerate}
\item $\d(\TT,A) - 1 \leq \d(\TT,B) \leq \d(\TT,A) + 1$,
\item
\begin{enumerate}
\item $\r(\TT,A) - 1 \leq \r(\TT,B) \leq \r(\TT,A)$
\item $\r(A,\TT) \leq \r(B,\TT) \leq \r(A,\TT) + 1$
\end{enumerate}
\end{enumerate}
\end{lemma}

\begin{proof}
Since $\Hom(A,B) \not= 0$, one finds 
\begin{eqnarray*}
0 &=& \r (A, B) \\
&\geq& \r (\TT, B) - \r (\TT, A) \\
&=& \left\lfloor \frac{\d (\TT, B)}{2}\right\rfloor - \left\lfloor \frac{\d (\TT, A)}{2}\right\rfloor.
\end{eqnarray*}
Hence $\r(\TT,B) \leq \r(\TT,A)$ and $\d(\TT,B) \leq \d(\TT,A) + 1$.  Likewise, one finds
\begin{eqnarray*}
0 &=& \r (A,B) \\
&\geq& \r (A,\TT) - \r (B,\TT) \\
&=& \left\lceil \frac{\d (\TT,A)}{2}\right\rceil - \left\lceil \frac{\d (\TT,B)}{2}\right\rceil
\end{eqnarray*}
so that $\r(A,\TT) \leq \r(B,\TT)$ and $\d(\TT,A)-1 \leq \d(\TT,B)$.  The required inequalities follow readily.
\end{proof}

\begin{lemma}\label{lemma:FiniteTLA}
For any $A \in \ind \QQ_T$, there is a finite subset $\TLA \subseteq \TT$ with the following property:
$$\forall B \in \ind \QQ_\TT: \r(A,B)= 0 \Rightarrow \r(B,\TT) = \r(B,\TLA).$$
\end{lemma}

\begin{proof}
Fix a $T_i \in \TT$ such that $\r(\TT,A) = \r(T_i,A)$.  Let $T_j \in \TT$ such that $\r(B,\TT) = \r(B,T_j)$ for some $B \in S_\QQ^\bullet(A,0)$.  If $i \not= j$, then using the triangle inequality we find
\begin{eqnarray*}
\r(\TT,A) + \r(B,\TT) &=& \r(T_i,A) + \r(B,T_j) \\
&=& \r(T_i,A) + \r(A,B) + \r(B,T_j)  \\
&\geq& \r(T_i,T_j) \geq \max\{i,j\}.
\end{eqnarray*}
By Lemma \ref{lemma:HomsInQ} we know that $\r(B,\TT) \leq \r(A,\TT)+1$ so that
$$\r(\TT,A) + \r(A,\TT) +1 \geq \max\{i,j\}.$$
This shows that $j$ is bounded and hence that $\TLA$ is finite.
\end{proof}

\begin{theorem}\label{theorem:MainTilting}
Let $\AA$ be a connected abelian hereditary category satisfying Serre duality with category of projectives $\QQ_\AA$.  Assume that $\bZ \QQ_\AA$ satisfies (*).  Then there is a hereditary section $\QQ_\TT$ in $\bZ \QQ_\AA$ which is a dualizing $k$-variety, and $\bZ \QQ_\TT = \bZ \QQ$.
\end{theorem}

\begin{proof}
Let $\QQ_\TT$ be a hereditary section as described in Construction \ref{construction:ChoosingTilt}.  We need to check that the two conditions of Proposition \ref{proposition:ConditionsToBeDualizing} are satisfied.  We will only prove the first part, the second part is shown dually.

Let $A \in \ind \QQ_\TT$ and divide the set of indecomposables $B \in \ind \QQ_\TT$ with $\r(A,B) = 0$ into subsets
$$\SS_{T,i} = \{B \in \ind \QQ \mid \r(A,B) = 0, \r(B,\TT) = \r(B,T) = i\}$$
where $i \in \bZ$, $T \in \TT$.  It follows from Lemmas \ref{lemma:HomsInQ} and \ref{lemma:FiniteTLA} that only finitely many of these subsets are nonempty.

For each of these nonempty subsets $\SS_{T,i}$ we will construct, in two steps, an object $C_{T,i} \in \QQ_\TT$ such that $\Hom(B,C_{T,i}) \not= 0$ when $B \in \SS_{T,i}$.  The object
$$C = \bigoplus_{\SS_{T,i} \not= \emptyset} C_{T,i}$$
is then the required object from the first condition of Proposition \ref{proposition:ConditionsToBeDualizing}.

Let $\QQ_A$ be the light cone centered on $A$ and let $\HH_A$ be an associated hereditary category in the sense of Theorem \ref{theorem:SectionTilting}, thus $\HH_A$ is the hereditary heart of a $t$-structure on $\Db \AA$ such that the category of projectives of $\HH_A$ correspond to $\QQ_A$.  In particular any $B \in \SS_{T,i}$ corresponds to a projective object in $\HH_A$ and because $\Hom(B,\t^{-i} T) \not= 0$ (due to Proposition \ref{proposition:LeftmostComposition}) we know that $\t^{-i} T \in \Ob \HH_A[0]$.  Moreover, since $\r(A,\t^{-i} T) \not= -\infty$, we know that $\t^{-i} T$ is even a preprojective object in $\HH_A$.  Proposition \ref{proposition:PreprojectivesCopresented} shows there is a projective cover $X \to \t^{-i} T$ in $\HH_A$.  Note that $\Hom(B,X) \not= 0$ for all $B \in \SS_{T,i}$.

Let $Y$ be a maximal direct summand of $X$ such that for every indecomposable direct summand $Y'$ of $Y$ there is a $B \in \SS_{T,i}$ with $\Hom(B,Y') \not= 0$, thus $\r(B,Y') = 0$.  Using the triangle inequality we find $\r(Y',\TT) = \r(Y',T)=i$, and $\r(\TT,Y') \geq \r(\TT,T) - \r(Y',T) = -i$.

In general the object $Y$ does not have to lie on $\QQ_\TT$.  In the second step of this construction, we will use the object $Y$ to construct the required object $C_{T,i}$.

Let $j \in \bZ$ be the smallest integer such that $\lfloor \frac{i+j}{2}\rfloor = j$, thus an object $Z' \in \ind \bZ \QQ_\TT$ with $j \leq \r(\TT,Z') \leq \r(\TT,B)$ and $\r(Z',\TT) = i$ would lie in the subcategory $\QQ_\TT$, if $B \in \ind \SS_{T,i}$.  Note that $j \leq \r(\TT,B)$ for all $B \in \SS_{T,i}$.

Let $\TT_f \subseteq \TT$ be the subset consisting of all objects $T_k \in \TT$ such that $\r(T_k, T) < i+j$.  Since $\TT$ satisfies the conditions of Lemma \ref{lemma:ChooseTT}, this is necessarily a finite set.  Note that the triangle inequality implies that any $T' \in \TT$ with $\r(T',Y') < j$ lies in $\TT_f$.  We now apply Lemma \ref{lemma:SubcategoryInMainProof} below to the hereditary section $\QQ_A$ with $i_1 = -i$ and $i_2 = j+1$.  We obtain a full subcategory $\QQ'_A$ of $\QQ_A$ and a right adjoint $r:\QQ_A \to \QQ_A'$ to the embedding.  We will write $Z$ for the maximal direct summand of $r(Y)$ such that for every indecomposable direct summand $Z'$ of $Z$ there is a $B \in \SS_{T,i}$ with $\Hom(B,Z') \not= 0$.  Note that $\Hom(B,Z) \not= 0$ for all $B \in \SS_{T,i}$.

We claim that every direct summand $Z'$ of $Z$ lies in $\QQ_\TT$.  Note that $\Hom(Z',r(Y)) \not= 0$ and thus $\Hom(Z',Y) \not= 0$.  This implies that $\r(Z',Y') \not= 0$, for a direct summand $Y'$ of $Y$, and thus the triangle inequality gives $\r(Z',\TT) \leq \r(Y',\TT) = i$.  There is also a $B \in \SS_{T,i}$ with $\r(B,Z') = 0$, and we use the triangle inequality to shows that $\r(Z',\TT) \geq i$.  We conclude that $\r(Z',\TT) = i$.

Next, Lemma \ref{lemma:SubcategoryInMainProof} yields that $j \leq \r(\TT_f, Z')$.  Since $\TT_f \subseteq \TT$ we know that $\r(\TT_f,Z') \leq \r(\TT,Z')$.  To proof that $j \leq \r(\TT, Z')$, let $T' \in \TT$ such that $\r(T',Z') < j$.  Then $\r(T',T) \leq \r(T',Z') + \r(Z',T) < j + i$ and thus by definition we have $T' \in \TT_f$.  This shows that indeed $j \leq \r(\TT, Z')$.

Since there is a $B \in \SS_{T,i}$ with $\r(B,Z') = 0$ we know that $\r(\TT,Z') \leq \r(\TT,B)$.  We conclude that 
$Z' \in \QQ_\TT$.  This shows that $Z$ is the required object $C_{T,i} \in \QQ$.
\end{proof}

We have used the following lemma.

\begin{lemma}\label{lemma:SubcategoryInMainProof}
Let $\QQ$ be a hereditary section in $\Db \AA$, and let $\TT_f  \subset \ind \bZ \QQ$ be a finite set.  Let $i_1, i_2 \in \bZ$ with $i_1 \leq i_2$.  There is a full subcategory $\QQ' \subseteq \QQ$ satisfying the following properties:
\begin{enumerate}
\item the embedding $\QQ' \to \QQ$ has a left and a right adjoint,
\item if $A \in \ind \QQ$ with $i_2 < \r(\TT,A)$, then $A \in \ind \QQ'$,
\item if $A \in \ind \QQ$ with $i_1 \leq \r(\TT,A) \leq i_2$, then $A \not\in \ind \QQ'$.
\end{enumerate}
\end{lemma}

\begin{proof}
Let $\HH$ be a hereditary heart corresponding to the hereditary section $\QQ$ as in the dual of Theorem \ref{theorem:SectionTilting}, thus such that $\QQ$ corresponds to the image of the category of injectives of $\HH$ into $\Db \AA$.

Write $\TT_f = \{T_0, T_1, \ldots, T_k\}$ and consider the set $\ZZ = \{\t^{j} T_i \mid T_i \in \TT_f, i_1 \leq j \leq i_2 \}$.  By possibly removing some elements from $\ZZ$, we may assume every element $Z \in \ZZ$ lies in $\HH \subset \Db \AA$.  Furthermore, every element of $\ZZ$ is directed so we can write $\ZZ = \{Z_0, Z_1, \ldots, Z_l\}$ such that $\Ext(Z_b,Z_a) = 0$ whenever $a \leq b$.

We define a full replete subcategory $\QQ'$ of $\QQ$ as follows:
$$A \in \Ob \QQ' \Leftrightarrow \forall Z \in \ZZ: \RHom(Z,A) = 0.$$

We prove that the category $\QQ'$ is the category from the statement of the lemma.  Note that Lemma \ref {lemma:LightConeDistanceAndShifts} implies that $\RHom(Z,A) = 0$ for all $Z \in \ZZ$ when $i_2 < \r(\TT_f,A)$, and that Proposition \ref{proposition:LeftmostComposition} implies that $\Hom(Z,A) \not= 0$ for some $Z \in \ZZ$ when $i_1 \leq \r(\TT_f,A) \leq i_2$.

Set $Z^{(0)} = Z_0$.  For $0 < a \leq l$ we define $Z^{(a)} = Z^{(a-1)} \oplus Z_{a}$ if $\Ext(Z_a, Z^{(a-1)}) = 0$, and by the universal extension
$$0 \to Z^{(a-1)} \to Z^{(a)} \to Z_a \otimes_k \Ext(Z_a, Z^{(a-1)}) \to 0$$
if $\Ext(Z_a, Z^{(a-1)}) \not= 0$.  It is straightforward to verify that $\Ext(Z^{(l)},Z^{(l)}) = 0$.  Since $A$ is an injective object in $\HH$, we now have
$$A \in \Ob \QQ' \Leftrightarrow \Hom(Z^{(l)},A) = 0.$$
The required result now follows from (the dual of) Proposition \ref{proposition:WhenAdjoints}.
\end{proof}

\section{Nonthread objects and threads in hereditary sections}\label{section:RoughsAndThreads}

\subsection{$\r$-in-between and threads}

As with dualizing $k$-varieties, the concepts of threads will be paramount in our discussion of hereditary sections.  However, a major difference between dualizing $k$-varieties and hereditary sections is that in the latter one can encounter so-called broken threads and a sort of half-open threads, called rays or corays.  To describe these cases, we start with a definition.

Let $\QQ$ be a hereditary section in $\Db \AA$ where $\AA$ is an abelian category with Serre duality and let $X,Y \in \ind \QQ$ with $\r(X,Y) < \infty$.  We will say $Z \in \ind \QQ$ is \emph{$\r$-in-between} $X$ and $Y$ if $\r(X,Y) = \r(X,Z) + \r(Z,Y)$.  We denote the full replete additive subcategory of $\QQ$ generated by all indecomposables $\r$-in-between $X$ and $Y$ by $[X,Y]^\bullet_\QQ$, thus
$$\ind [X,Y]_\QQ^\bullet = \{Z \in \ind \QQ \mid \r(X,Z) + \r(Z,Y) = \r(X,Y)\}.$$
If there is no confusion, we will often write $[X,Y]^\bullet$ instead of $[X,Y]_\QQ^\bullet$.

We will define $]X,Y]_\QQ^\bullet$ to be the full replete additive subcategory of $[X,Y]_\QQ^\bullet$ spanned by the objects not supported on $X$.  Likewise one defines $[X,Y[^\bullet_\QQ$ and $]X,Y]^\bullet_\QQ$

\begin{remark}\label{remark:IntervalsCoincide}
If $\r(X,Y) = 0$, then $[X,Y]^\bullet_\QQ = [X,Y]$.
\end{remark}

\begin{example}\label{example:rInBetween}
Let $Q$ be the quiver given by
$$\xymatrix{&b \ar[d] \ar[rd]& \\
a \ar[ru] \ar[r] \ar[rd] & c & e \\
&d \ar[u] \ar[ru]&}$$
Denote by $P_i$ the projective object in $\rep Q$ associated with the vertex $i$ of $Q$.  Let $\QQ$ be the standard hereditary section in $\Db \rep Q$.  We have
\begin{eqnarray*}
\ind [P_a,P_e]_\QQ^\bullet &=& \{P_a,P_b,P_d,P_e\} \\
\ind [P_b,P_d]_\QQ^\bullet &=& \{P_a,P_b,P_c,P_d,P_e\} \\
\ind [P_a,P_d]_\QQ^\bullet &=& \{P_a,P_d\} \\
\ind [P_d,P_a]_\QQ^\bullet &=& \{P_a,P_d,P_c\}
\end{eqnarray*}
Note that $\ind [P_b,P_d]^\bullet \not\subseteq \ind [P_a,P_e]^\bullet$ (even though $P_b,P_d \in \ind [P_a,P_e]^\bullet$) and that $[P_a,P_d]^\bullet \not\cong [P_d,P_a]^\bullet$.
\end{example}

\begin{remark}
As the previous example indicates, the subcategories $[X,Y]^\bullet$ are the replacement of $\r$-geodesics on a quiver.
\end{remark}

\begin{proposition}\label{Proposition:SameOrbits}
Let $X,Y \in \QQ$ with $n=\r(X,Y) < \infty$.  The sets $\ind [X,Y]^\bullet_\QQ$ and $\ind [X, \t^{-n} Y]$ intersect the same $\t$-orbits.
\end{proposition}

\begin{proof}
This follows immediately from Lemma \ref{lemma:LightConeDistanceAndShifts}.
\end{proof}

\begin{example}
Let $Q$ be the quiver from Example \ref{example:rInBetween}.  The light cones centered on $P_a, P_b$ and $P_d$ are given by
$$\xymatrix{&P_b \ar[d] \ar[rd]&&&P_b \ar[d] \ar[rd] \ar[ld]&&& \tm P_b & \\
P_a \ar[ru] \ar[r] \ar[rd] & P_c & P_e & \tm P_a \ar[rd] & P_c \ar[d] \ar[l] & P_e \ar[ld]& \tm P_a \ar[ru] & P_c \ar[u] \ar[l] & P_e \ar[lu] \\
&P_d \ar[u] \ar[ru]&&& \tm P_d &&& P_d \ar[lu] \ar[u] \ar[ru]}$$
\end{example}

\begin{corollary}\label{corollary:HalfConvex}
Let $X,Y,Z \in \ind \QQ$ with $\r(X,Y) < \infty$ and $Z \in [X,Y]^\bullet$.  In this case $[X,Z]^\bullet \subseteq [X,Y]^\bullet$ and $[Z,Y]^\bullet \subseteq [X,Y]^\bullet$.   
\end{corollary}

\begin{definition}
Let $X \in \ind \QQ$.  It follows from Observation \ref{observation:PropertiesSection} that $\QQ$ has right and left almost split maps.  Let $X \to M$ and $N \to X$ be a left and right almost split map, respectively.  If $M$ and $N$ are indecomposable, we will say $X$ is a \emph{thread object}.  We will denote $M$ and $N$ by $X^+$ and $X^-$, respectively.  An object which is not a thread object is called a \emph{nonthread object}.

If $[X,Y]^\bullet$ consists of only thread objects in $\QQ$, then we call $[X,Y]^\bullet$ a \emph{thread}\index{thread}.  If furthermore $\r(X,Y) > 0$ or $\r(X,Y)=0$, then we call $[X,Y]^\bullet$ a \emph{broken thread}\index{thread!broken} or an \emph{unbroken thread}\index{thread!unbroken}, respectively.
\end{definition}

\begin{example}\label{example:BrokenThread}
Let $\aa = kQ$ where $Q$ is the thread quiver $\xymatrix@1{\cdot\ar@{..>}[r]&\cdot}$  Thus $\aa = k(\bN \cdot -\bN)$ and the indecomposable projectives of $\mod \aa$ are given by $\aa(-,n)$ and $\aa(-,-n)$ for $n \in \bN$.

The Auslander-Reiten quiver of $\Db \mod \aa$ may be sketched as in the upper part of Figure \ref{fig:BrokenThread} where the triangles represent $\bZ A_\infty$-components and where the category $\mod \aa$ has been marked with gray.

We will denote by $\QQ$ the hereditary section in $\Db \mod \aa$ corresponding to the projectives of $\mod \aa$.  The interval $[\aa(-,1),\aa(-,-1)]=[\aa(-,1),\aa(-,-1)]^\bullet \subset \ind \QQ$ is an (unbroken) thread.

\begin{figure}[tb]
	\centering
		\includegraphics[width=0.75\textwidth]{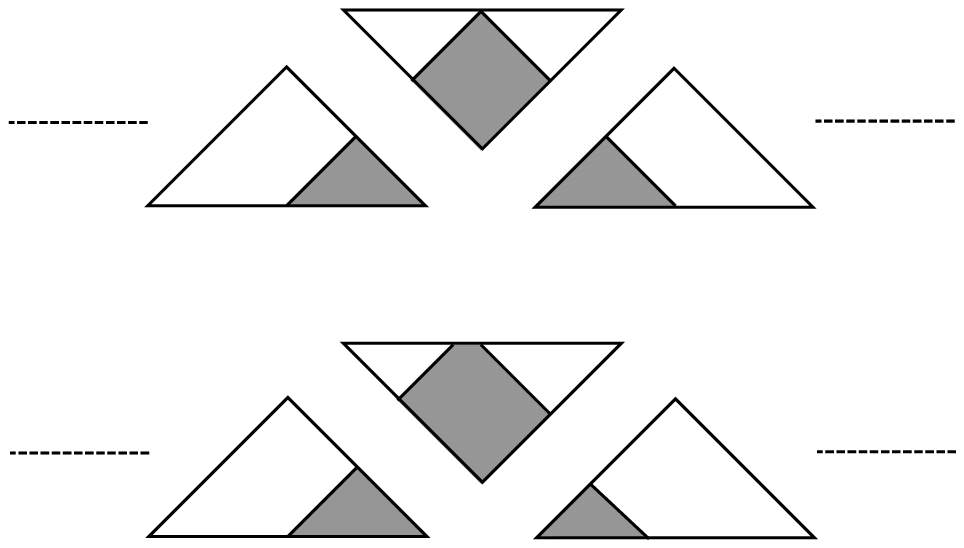}
	\caption{Illustration of Example \ref{example:BrokenThread}}
	\label{fig:BrokenThread}
\end{figure}

Consider the hereditary section $\QQ' \subseteq \Db \mod \aa$ spanned by all objects of the form $\aa(-,n)$ and $\tau \aa(-,-n)$ where $n \in \bN$ as in the lower part of Figure \ref{fig:BrokenThread}.  Now $[\aa(-,1), \tau \aa(-,-1)]^\bullet \subset \ind \QQ'$ is a broken thread.
\end{example}

A reason to introduce thread objects is given by the following observation: let $X,Y \in \ind \QQ$ and consider the left adjoint $l$ to the embedding $i:[X,Y] \to \QQ$ (see Proposition \ref{proposition:WhenAdjoints}).  Let $A$ be any indecomposable object in $\QQ$.  If $A$ does not lie in $[X,Y]$, then the only thread object which can occur as a direct summand of $l(A)$ is $X$.  Indeed, let $Z$ be a thread object which is a direct summand of $l(A)$.  If $Z \not\cong X$ then $Z^- \in \Ob [X,Y]$.  Since $\QQ$ is semi-hereditary, we know that $\dim \Hom(l(A),Z) > \dim \Hom(l(A),Z^-)$.  However, since no map $A \to Z$ is a split map, we have $\dim \Hom(A,iZ) = \dim \Hom(A,iZ^-)$.  A contradiction.

This proves the following lemma.

\begin{lemma}\label{lemma:ThreadsAndAdjoints}
Let $X,Y \in \ind \QQ$ and let $l$ be a left adjoint to the embedding $i:[X,Y] \to \QQ$.  A thread object in $]X,Y]$ cannot be a direct summand of $l(A)$, for any $A \in \ind \QQ \setminus \ind [X,Y]$.
\end{lemma}

\begin{example}
Let $Q$ be the quiver
$$\xymatrix{
&f\ar[ld]\ar[rd] \\
a \ar[r]& b \ar[r]& c \ar[r]& d \ar[r]& e}$$
Denote by $P_i \in \ind \rep Q$ an indecomposable projective associated with the vertex $i$ of $Q$, let $\QQ$ be the standard hereditary section in $\Db \rep \QQ$, and let $l: \QQ \to [P_b,P_d]$ be the left adjoint to the embedding $[P_b, P_d] \to \QQ$.  We have that $l(P_f) \cong P_b \oplus P_c$.
\end{example}

Let $A,B \in \ind \QQ$.  While the subcategory $[A,B]$ will only be nontrivial if $\Hom(A,B) \not= 0$, a similar statement is not true for $[A,B]^\bullet$.  In fact, as Remark \ref{remark:IntervalsCoincide} indicates we will mostly be interested in cases where $\r(A,B) \not= 0$.  This means however, as the following example shows, that we can encounter situations where we consider $[A,B]^\bullet$ where $\r(B,A) = 0$.

\begin{example}
Let $Q$ be the quiver $A_5$ with linear orientation, thus $Q$ is given by
$$\xymatrix@1{a \ar[r]& b \ar[r]& c \ar[r]& d \ar[r]& e}$$
Denote by $P_i \in \ind \rep Q$ an indecomposable projective associated with the vertex $i$ of $Q$, and let $\QQ$ be the standard hereditary section in $\Db \rep \QQ$.  We have that $[P_d,P_b]^\bullet = [P_b,P_d]$.
\end{example}

In some sense the interval $[P_d,P_b]^\bullet$ from the previous example does not have the ``natural'' orientation.  The following lemma and Proposition \ref{Proposition:InbetweenObjects} below indicate that we can look at the neighbors of $P_b$ and $P_d$ to somewhat compensate for this lack of orientation.

\begin{lemma}\label{lemma:Neighbors}
Let $\QQ$ be a hereditary section, and let $A,B \in \ind \QQ$ with $\r(A,B) < \infty$.
\begin{enumerate}
\item If $A \not= B$, then $A$ and $B$ have a least one neighbor lying in $[A,B]^\bullet$.  An $X \in \ind [A,B]^\bullet$ with $A \not= X \not= B$ has at least two (non-isomorphic) direct neighbors in $[A,B]^\bullet$.
\item Assume $[A,B]^\bullet$ is a thread (with $A \not= B$) in $\QQ$.  If $B^- \not\in [A,B]^\bullet$, then $A^-,B^+ \in \ind [A,B]^\bullet$ and $\r(B,A) < \infty$.
\end{enumerate}
\end{lemma}

\begin{proof}
The first result follows immediately from Proposition \ref{Proposition:SameOrbits}.  For the second result, let $B=B_0 \to B_1 \to B_2 \to \cdots$ be a (possibly finite) sequence of direct successors.  Since $B^-$ does not lie in $[A,B]^\bullet$, and $B$ is a thread object, we know $B_1$ lies in $[A,B]^\bullet$.  If $B_1 \not= A$, then it is a thread object and we know $B_2$ also lies in $[A,B]^\bullet$.

Iteration shows either the entire sequence lies in $[A,B]^\bullet$, or some $B_i=A$.  Since $\r(B_{j+1},B_{j}) = 1$, we find that $\r(B_j,B) = j$.  This shows that $B_i = A$ where $i = \r(A,B)$ and that both $B^+ = B_1$ and $A^- = B_{i-1}$ lie in $[A,B]^\bullet$.
\end{proof}

\begin{example}
Let $Q$ be the quiver given by
$$\xymatrix@1{a \ar[r]& b \ar[r] & c \ar[r] & d \ar[r] & e}$$
Denote by $P_i \in \ind \rep Q$ an indecomposable projective associated with the vertex $i$ of $Q$, and let $\QQ$ be the standard hereditary section in $\Db \rep \QQ$.  Since we have $\ind [P_b,P_d]^\bullet = \{P_b,P_c,P_d\} = \ind [P_d,P_b]^\bullet$, both $[P_b,P_d]^\bullet$ and $[P_d,P_b]^\bullet$ are threads.  We easily see that the results of the previous lemma are valid in this case.

If we replace the quiver $Q$ by
$$\xymatrix@1{a \ar[r]& b \ar[r] & c & d \ar[l] & e \ar[l]}$$
then, with the same notations as above, $\ind [P_d,P_b]^\bullet = \{P_b,P_c,P_d\}$ would not be a thread.  Note that $P_b^- \not \in [P_d,P_b]^\bullet$, but also $P_d^- \not\in [P_d,P_b]^\bullet$.
\end{example}

The following proposition resembles Proposition \cite[Proposition 6.2]{BergVanRoosmalen09}.

\begin{proposition}\label{Proposition:InbetweenObjects}
Let $[X,Y]^\bullet$ be a thread.  If $[X,Y]^\bullet$ and $[X,Y']^\bullet$ share an indecomposable apart from $X$, then $[X,Y]^\bullet \subseteq [X,Y']^\bullet$ or $[X,Y']^\bullet \subseteq [X,Y]^\bullet$.
\end{proposition}

\begin{proof}
We will work in the light cone $\QQ_X$ centered on $X$.  Write $n= \r(X,Y)$ and $n'=\r(X,Y')$.  The assumption in the statement shows there is a $Z \in \ind ]X,\t^{-n}Y] \cap \ind ]X,\t^{-n'}Y']$.  Note that $]X,\t^{-n}Y]$ is a thread.

Proposition \ref{proposition:WhenAdjoints} yields that the embedding $[X,\t^{-n}Y] \to \QQ_X$ has a right adjoint $r: \QQ \to [X,\t^{-n}Y]$.  Since $Z \in \ind ]X,\t^{-n}Y]$ and $\Hom(iZ,\t^{-n'}Y') \not= 0$, we find that $r(\t^{-n'}Y')$ has a nonzero direct summand lying in $]X,\t^{-n}Y]$.

Lemma \ref{lemma:ThreadsAndAdjoints} yields that either $\t^{-n'}Y' \in \Ob ]X,\t^{-n}Y[$ and thus $[X,\t^{-n'}Y'] \subseteq [X,\t^{-n} Y]$ by Corollary \ref{corollary:HalfConvex}, or that $\Hom(\t^{-n}Y, r(\t^{-n'}Y')) \not= 0$ and thus $\t^{-n}Y \in \Ob ]X,\t^{-n'}Y']$ so that $[X,\t^{-n}Y] \subseteq [X,\t^{-n'} Y']$.

Applying Proposition \ref{Proposition:SameOrbits} shows the required property.
\end{proof}

\begin{example}
Let $Q$ be the quiver $A_5$ with linear orientation, thus $Q$ is given by
$$\xymatrix@1{a \ar[r]& b \ar[r]& c \ar[r]& d \ar[r]& e}$$
Denote by $P_i \in \ind \rep Q$ an indecomposable projective associated with the vertex $i$ of $Q$, and let $\QQ$ be the standard hereditary section in $\Db \rep \QQ$.  The threads $[P_c, P_b]^\bullet$ and $[P_c,P_d]^\bullet$ do have $P_c$ in common, but no other indecomposable.  Neither thread is a subcategory of the other such that the result from Proposition \ref{Proposition:InbetweenObjects} does not hold.
\end{example}

\subsection{Nonthread objects}\label{subsection:Roughs}

In this subsection, we will give a short discussion of the nonthread objects of $\QQ$.  Our main result will be that, if $\bZ \QQ$ is connected, $\QQ$ has only countably many nonthread objects.

\begin{lemma}\label{lemma:FewRoughsInIntervals}
Let $\QQ$ be a hereditary section in $\Db \AA$ and let $X \in \ind \QQ$.
\begin{enumerate}
\item For every $Y \in \ind \QQ$ with $\r(X,Y)=0$, there are only finitely many nonthread objects in $[X,Y]$.
\item For every $Y \in \ind \QQ$ with $\r(X,Y)\in \bZ$, there are only finitely many nonthread objects in $[X,Y]^\bullet$.
\item Assume $X$ is a nonthread object.  For every $Y \in \ind \QQ$ with $\r(X,Y)\in \bZ$, there is a nonthread object $Z \in [X,Y]^\bullet$ such that $]Z,Y[^\bullet$ has no nonthread objects.
\end{enumerate}
\end{lemma}

\begin{proof}
\begin{enumerate}
\item Let $A$ be a nonthread object in $[X,Y]$.  If $A$ is not isomorphic to $X$ or $Y$, then Lemma \ref{lemma:Neighbors} implies there are (nonzero) almost split maps $N_A \to A$ and $A \to M_A$ in $\QQ$.  Since $A$ is a nonthread object, either $M_A$ or $N_A$ is not indecomposable.  Seeking a contradiction, assume there are infinitely many nonthread objects $A$ such that $N_A$ is not indecomposable.

Let $\AA$ be a heart of a $t$-structure associated with $\QQ$ as in Theorem \ref{theorem:SectionTilting}.  We denote $Z = \im (Y \to \bS X \otimes \Hom(Y,\bS X)^*)$ and $K = \ker(Y \to \bS X \otimes \Hom(Y,\bS X)^*)$.

Since $\QQ$ is semi-hereditary (Observation \ref{observation:PropertiesSection}) and $\dim \Hom(X,Y) < \infty$, there can only be finitely many objects $A \in \ind [X,Y]$ such that $\dim \Hom(X,N_A) > \Hom(X,A)$.  There are hence infinitely many objects $A \in \ind [X,Y]$ such that $\dim \Hom(X,N_A) = \Hom(X,A)$.  Any direct summand $B$ of $N_A$ not lying in $[X,Y]$ is necessarily a direct summand of $K$, but $K$ is a finitely generated projective object.  We conclude that there are infinitely many nonthread objects $A \in [X,Y]$ such that $N_A$ is not indecomposable.

Likewise, one shows there are only finitely many nonthread objects $A \in [X,Y]$ such that $M_A$ is not indecomposable.
  
\item   Seeking a contradiction, assume $[X,Y]^\bullet_\QQ$ has infinitely many nonthread objects in $\QQ$.  Denote $n=\r(X,Y)$ and let $\QQ_X$ be the light cone centered on $X$.  It follows from the previous part that $[X,\t^{-n} Y]$ has only finitely many nonthread objects in $\QQ_X$, thus infinitely many nonthread objects in $[X,Y]^\bullet_\QQ$ are either a sink or a source with exactly two direct neighbors.  Denote by $\{A_i\}_{i \in I} \subseteq \ind [X,Y]^\bullet_\QQ$ such an infinite set of sinks and sources, and denote by $A_i'$ the object in $\ind [X,\t^{-n} Y]$ lying in the same $\t$-orbit as $A_i$ (see Proposition \ref{Proposition:SameOrbits}).  We define a partial ordering on $I$ by $i \leq j \Leftrightarrow \Hom(A'_i,A'_j)\not=0$.

Since $\QQ_X$ is semi-hereditary (Observation \ref{observation:PropertiesSection}) and $\dim \Hom(X,\t^{-n} Y) < \infty$, we know that there is an infinite linearly ordered subposet $J$ of $I$.  Furthermore either infinitely many elements of $\{A_j\}_{j \in J}$ are sinks or infinitely many are sources.

If $A_j$ is a sink, then $\r(A_j,A_k) > 0$ when $j,k \in J$ with $j<k$.  Also note that, since $A'_j \in [X,A'_k]$, Proposition \ref{Proposition:SameOrbits} shows that $A_j \in [X,A_k]^\bullet$ and thus $\r(X,A_k) = \r(X,A_j) + \r(A_j,A_k)$.  We infer that $\r(X,A_j) > \r(X,A_k)$ for any $k > j$ and hence $\{A_j\}_{j \in J}$ cannot have infinitely many sinks.

Likewise one shows that $\{A_j\}_{j \in J}$ cannot have infinitely many sources.  A contradiction.

\item Take a nonthread object $Z \in \ind [X,Y[^\bullet$ such that $[Z,Y[^\bullet$ has a minimal number of nonthread objects.  Using Corollary \ref{corollary:HalfConvex} it is easy to see that $Z$ is the only nonthread object in $[Z,Y[^\bullet$.
\end{enumerate}
\end{proof}

\begin{lemma}\label{lemma:CountablyInFiniteDistance}\label{lemma:RoughZigZag}
Let $\QQ$ be a hereditary section in $\Db \AA$ and let $X \in \ind \QQ$.
\begin{enumerate}
\item There are only finitely many nonthread objects $Y$ such that $]X,Y[$ is nonempty and has no nonthread objects.
\item There are only countably many nonthread objects $Y \in \ind \QQ$ with $\r(X,Y) = 0$.
\item There are only countably many nonthread objects $Y \in \ind \QQ$ with $\r(X,Y) < \infty$.
\end{enumerate}
\end{lemma}

\begin{proof}
\begin{enumerate}
\item Let $X \to M$ be a left almost split map in $\QQ$.  Let $M'$ be an indecomposable summand of $M$ and let $Y_1,Y_2$ be two nonisomorphic nonthread objects with $\Hom(M',Y_1) \not= 0$ and $\Hom(M',Y_2) \not= 0$ such that $]X,Y_1[$ and $]X,Y_2[$ are nonempty and have no nonthread objects.  Note that this implies that $Y_1 \not\cong M' \not\cong Y_2$.  In particular we know that $M'$ is a thread object in $\QQ$.

The embedding $i:[M',Y_1] \to \QQ$ has a right adjoint $i_R$.  Since $\Hom(Y_1,Y_2) = 0$ we know that the object $i \circ i_R(Y_2)$ lies in $[M',Y_1[$.  However, every indecomposable direct summand of $i \circ i_R(Y_2)$ is a thread object in $\QQ$, contradicting Lemma \ref{lemma:ThreadsAndAdjoints}.

\item Denote by $N_i^X$ the set of all nonthread objects $Y \ind \QQ$ such that $\r(X,Y) = 0$ and $]X,Y[$ has exactly $i$ nonthread objects.  Lemma \ref{lemma:FewRoughsInIntervals} yields that it is sufficient to prove that the set $\cup_{i \in \bN} N_i^X$ is countable.

It was shown above that $N^X_0$ is finite for all $X \in \ind \QQ$; we will proceed by induction.  Assume therefore that $N_j^Z$ is finite for every $j < i$ and every $Z \in \ind \QQ$.  We will prove that $N_{i}^X$ is finite.  Let $Y \in N_{i}^X$ and let $Z$ be a nonthread object in $]X,Y[$, thus $Z \in N_j^X$ for some $j < i$.  Corollary \ref{corollary:HalfConvex} yields that $Y \in N_k^Z$ for some $k < i$ so that
$$N_i^X \subseteq \bigcup_{j,k < i} \bigcup_{Z \in N_j^X} N_k^Z.$$
Since the right hand side is a finite union of finite sets, the left hand side is finite as well.  This shows that the set $\cup_{i \in \bN} N_i^X$ is countable.

\item We will prove there are only countably many nonthread objects $Y$ with $\r(X,Y) = n$.  Seeking a contradiction, assume there are uncountably many such nonthread objects.

Let $\QQ_X$ be the light cone centered on $X$.  Every nonthread object $Y \in \QQ$ with $\r(X,Y) = n$ corresponds to an object $Y' = \t^{-n} Y \in \QQ_X$.  It follows from the previous part that $\QQ_X$ has only countably many nonthread objects, such that uncountably many nonthread objects $Y \in \QQ$ with $\r(X,Y) = n$ correspond to thread objects $Y' \in \QQ_X$.  In particular, we know that $Y$ has either exactly two (nonisomorphic) direct predecessors or direct successors in $\QQ$.

We define a new hereditary section $\QQ_n$ generated by the indecomposables $\t^{-m_A} A$ where $A \in \ind \QQ$ and $m_A = \min(\r(X,A),n)$.  Note that $\QQ_0 = \QQ$.  To prove that $\QQ_n$ is indeed a hereditary section, it suffices to show that $\r(\t^{-m_A} A,\t^{-m_B} B) \geq 0$ for all $A,B \in \ind \QQ$ (see Corollary \ref{corollary:TauTilting}).  We have
\begin{eqnarray*}
\r(\t^{-m_A} A,\t^{-m_B} B) &=& \r(A,B) +m_A - m_B \\
&\geq& \r(A,B) + \r(X,A) - m_B \\
&\geq& \r(X,B) - m_B \geq 0
\end{eqnarray*}

Let $Y \in \ind \QQ$ be a nonthread object with $\r(X,Y) = n$.  If $Y$ has two direct predecessors $M_1,M_2$ in $\QQ$, then it follows from the triangle inequality that $\r(X,M_1) \geq n$ and $\r(X,M_2) \geq n$ such that $Y'$ is a nonthread object in $\QQ_n$.  Since $X$ and $Y'$ both lie in $\QQ_n$ and $\r(X,Y') = 0$, we know there are only countably many of such objects.

Consider the case where $Y$ has two direct successors $N_1,N_2$ in $\QQ$.  For any direct successor $N_i$ we have $\r(Y,N_i) = 0$ and $\r(N_i,Y) = 1$ such that the triangle inequality shows that either $\r(X,N_i) = n$ or $\r(X,N_i) = n-1$.

If both $\r(X,N_1) = n$ and $\r(X,N_2) = n$, then $Y'$ is a nonthread object in $\QQ_n$ with $\r(X,Y')$ and we know there are only countably many of such objects.  We may thus assume that $\r(X,N_1) = n-1$.  In this case $\t^{-n+1} N_1$ is a nonthread object in $\QQ_{n-1}$ since $\t^{-n+1} N_1$ has at least two nonisomorphic direct predecessors: $\t^{-n+1} Y$ and one lying in $[X,\t^{-n+1} N_1]$.  Thus $\t^{-n+1} Y$ is a direct neighbor of a nonthread object $\t^{-n+1} N_1$ in $\QQ_{n-1}$, and there can again only be countably many of these objects.

\end{enumerate}
\end{proof}

\begin{proposition}\label{proposition:NonthreadPath}
Let $\AA$ be an abelian category with Serre duality and let $\QQ$ be a hereditary section in $\Db \AA$.  If $\bZ \QQ$ is connected then for every $X,Y \in \ind \QQ$ with $\d(X,Y)$ there is a sequence $X=X_0, X_1, X_2, \ldots, X_n = Y$ in $\ind \QQ$ such that
\begin{enumerate}
\item either $\r(X_i,X_{i+1}) < \infty$ or $\r(X_{i+1},X_i) < \infty$ for all $0 \leq i \leq n-1$, and
\item $\r(X_i,X_{i+2}) = \r(X_{i+2},X_i) = \infty$, for all $0 \leq i \leq n-2$, and
\item the objects $X_i$ are nonthread objects in $\QQ$, for $1 \leq i \leq n-1$.
\end{enumerate}
\end{proposition}

\begin{proof}
The existence of the sequence satisfying the first two properties follows from the connectedness of $\bZ \QQ$ and the triangle inequality, so we need only to prove the last property.  It suffices to prove the following statement: let $X,Y,Z \in \ind \QQ$ with $\r(X,Z)<\infty$ and $\r(Y,Z)<\infty$.  If $\r(X,Y) = \infty$ and $\r(Y,X)=\infty$, then there is a nonthread object $Z' \in \ind \QQ$ with $\r(X,Z')<\infty$ and $\r(Y,Z')<\infty$.

We may assume all $A \in \ind \QQ$ with $\r(X,A) < \infty$ and $\r(Y,A) < \infty$ are thread objects.  Let $A$ be such a thread object such that $\r(X,A) + \r(Y,A) \geq 0$ is minimal.  Using the triangle inequality, we see that $\r(X,A) = \r(X,B)$ and $\r(Y,A) = \r(Y,B)$ for $B = A^+$.  This implies that $A \in \ind [X,B]^\bullet$ and $A \in \ind [Y,B]^\bullet$.  Proposition \ref{Proposition:InbetweenObjects} shows $A$ lies in the co-light cone $\QQ^B$ centered on $B$.

The embedding $[\t^{\r(X,B)} X, B] \to \QQ^B$ has a left adjoint $l: \QQ^B \to [\t^{\r(X,B)} X, B]$.  Consider the object $l(\t^{\r(Y,B)} Y)$.  By our initial assumption, every direct summand of $l(\t^{\r(Y,B)} Y)$, except possibly $B$, will be a thread object in $\QQ^B$.

However, there is at least one indecomposable direct summand which maps nonzero to $A$.  Lemma \ref{lemma:ThreadsAndAdjoints} shows that either $\t^{\r(Y,B)} Y \in [\t^{\r(X,B)} X, B]$ or $\Hom(\t^{\r(Y,B)} Y,\t^{\r(X,B)} X) \not= 0$.  This contradicts $\r(X,Y) = \infty$ or $\r(Y,X) = \infty$, respectively.
\end{proof}

\begin{proposition}\label{proposition:NonthreadStructure}
Let $\AA$ be an abelian category with Serre duality and let $\QQ$ be a hereditary section in $\Db \AA$.  If $\bZ \QQ$ is connected, then $\QQ$ has only countably many nonthread objects.
\end{proposition}

\begin{proof}
If $\QQ$ has no nonthread objects, then the statement is trivial.  Thus assume $\QQ$ has at least one nonthread object $X$.  It follows from Proposition \ref{proposition:NonthreadPath} that for every nonthread object $Y \in \ind \QQ$ there is a sequence $X=X_0, X_1, X_2, \ldots, X_n = Y$ in $\ind \QQ$ with either $\r(X_i,X_{i+1}) < \infty$ or $\r(X_{i+1},X_i) < \infty$ such that $X_i$ is a nonthread object when $0<i\leq n$.  Lemma \ref{lemma:CountablyInFiniteDistance} then yields that there are only countably many nonthread objects in $\QQ$.
\end{proof}

\subsection{Rays and corays}\label{subsection:Rays}

Let $\AA$ be an abelian Ext-finite category with Serre duality and $\QQ$ a hereditary section in $\Db \AA$.  Let $\TT \ind \QQ$ be the subset of all nonthread objects.  An object $X \in \ind \QQ$ is called a \emph{ray object} if $\r(X,\TT) = \infty$ and is called a \emph{coray object} if $\r(\TT,X) = \infty$.

Note that ray objects and coray objects are always thread objects.

We will define an equivalence relation on ray objects as follows: two ray objects $X,Y \in \ind \QQ$ are equivalent if and only if $\r(X,Y) < \infty$ or $\r(Y,X) < \infty$.  Reflexivity and symmetry are clear, while transitivity follows from the following lemma.

\begin{lemma}
Let $\QQ$ be a hereditary section and let $X,Y,Z \in \ind \QQ$ be ray objects.
\begin{itemize}
\item If $\r(X,Y) < \infty$ and $\r(Y,Z) < \infty$, then $\r(X,Z) < \infty$,
\item if $\r(Y,X) < \infty$ and $\r(Z,Y) < \infty$, then $\r(Z,X) < \infty$,
\item if $\r(X,Y) < \infty$ and $\r(Z,Y) < \infty$, then $\r(X,Z) < \infty$ or $\r(Z,X) < \infty$,
\item if $\r(Y,X) < \infty$ and $\r(Y,Z) < \infty$, then $\r(X,Z) < \infty$ or $\r(Z,X) < \infty$.
\end{itemize}
\end{lemma}

\begin{proof}
The first two cases follow from the triangle inequality.  For the third case, we may assume that both $\r(X,Y) = \infty$ and $\r(Y,Z) = \infty$ as otherwise the required result would follow directly from the triangle inequality.  It follows from Lemma \ref{lemma:Neighbors} that $Y^- \in [X,Y]^\bullet_\QQ$ and $Y^- \in [Z,Y]^\bullet_\QQ$.  Proposition \ref{Proposition:InbetweenObjects} yields the required result.

The last case is similar.
\end{proof}

The full replete additive subcategory of $\QQ$ generated by an equivalence class of ray objects is called a \emph{ray}.  A ray is necessarily infinite.

Dually, we define a coray as the full replete additive category generated by a maximal set of coray objects such that for any two objects $X,Y$ either $\r(X,Y) < \infty$ or $\r(Y,X) < \infty$.

\begin{example}
The hereditary section in Example \ref{example:NotStar} has a ray, and the hereditary section in Example \ref{example:NotStar2} has a coray.
\end{example}

\begin{example}\label{example:Halfopen2}
Let $\aa$ be the dualizing $k$-variety given by
$$\xymatrix@1{A_0\ar[r]&A_1\ar[r]&\ar@{..}[r]&\ar[r]&B_{-1}\ar[r]&B_0 \ar[r]& B_1 \ar[r]&\ar@{..}[r]&\ar[r]& C_{-1} \ar[r]&C_0 }$$
or, equivalently, given by the thread quiver $\xymatrix@1{\cdot \ar@{..>}[r]^{1}&\cdot}$.  The category $\Db \mod \aa$ may be sketched as the first part of Figure \ref{fig:Halfopen2} where the triangles represent $\bZ A_\infty$-components, and the squares represent $\bZ A_\infty^\infty$-components.  As usual, the abelian category $\mod \aa \subset \Db \mod \aa$ has been marked with gray.

\begin{figure}[tb]
	\centering
		\includegraphics[width=0.50\textwidth]{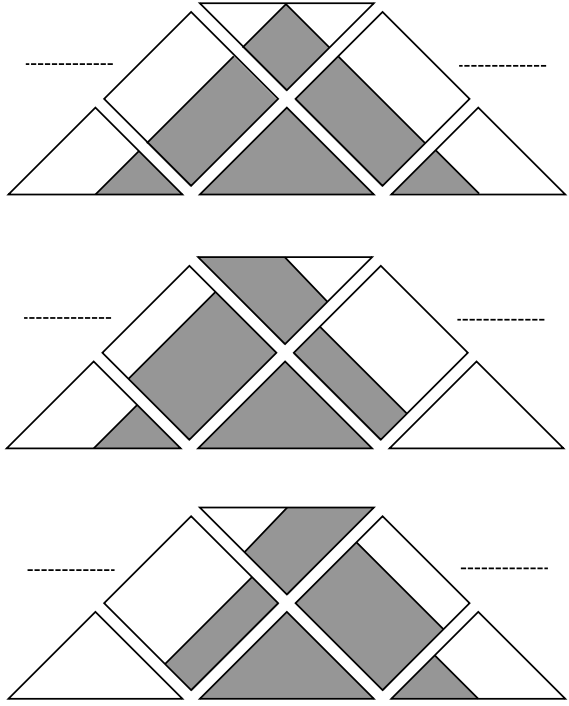}
	\caption{Illustration of Example \ref{example:Halfopen2}}
	\label{fig:Halfopen2}
\end{figure}

Choosing a hereditary section $\QQ$ spanned by objects of the form $\aa(-,A_i)$ and $\tau \aa(-,B_i)$ gives rise to an abelian category as sketched in the second part of Figure \ref{fig:Halfopen2}.  Here, $\QQ$ has a ray, but no coray.

Likewise, choosing a hereditary section $\QQ$ spanned by objects of the form $\tm\aa(-,B_i)$ and $\tau \aa(-,C_i)$ gives rise to an abelian category as sketched in the last part of Figure \ref{fig:Halfopen2}.  We see that $\QQ$ has a coray, but no ray.
\end{example}

Note that if $\QQ$ does not have any nonthread objects (thus $\TT = \emptyset$), then $\r(\TT,X) = \r(X,\TT) = \infty$ for every $X \in \ind \QQ$ so that every object is both ray object and a coray object.  If $\bZ \QQ$ is furthermore connected, then Proposition \ref{proposition:NonthreadPath} shows there is only one ray and one coray, and both coincide with $\QQ$.  Conversely it follows from Proposition \ref{proposition:NonthreadPath} that if $\bZ \QQ$ is connected and has an object which is both a ray and a coray object then $\QQ$ has no nonthread objects.

In light of Proposition \ref{proposition:NonthreadStructure}, the following observation is obvious.

\begin{observation}\label{observation:NeedsRays}
If $\QQ$ is a connected hereditary section without rays or corays, then $\QQ$ satisfies the condition (*).
\end{observation}

Let $\RR$ be a ray.  We will say $A \in \ind \QQ$ is an \emph{anchor} of $\RR$ if $A$ is a nonthread object with $\r(A,X) < \infty$ for all $X \in \ind \RR$, and it is the only nonthread object in $[A,X]^\bullet_\QQ$.  Dually, a coancher $A'$ of a coray $\RR'$ is defined to be a nonthread object with $\r(X',A') < \infty$ for all $X' \in \ind \RR'$, and it is the only nonthread object in $[X',A']^\bullet_\QQ$.

\begin{proposition}\label{proposition:RayOrdered}
Let $\QQ$ be a hereditary section with at least one nonthread object.  Let $X,Y \in \ind \QQ$ be elements of the same ray $\RR$ and let $A \in \ind \QQ$ be a nonthread object with $\r(A,X) < \infty$.  Then we also have that $\r(A,Y) < \infty$, and either $[A,X]^\bullet \subseteq [A,Y]^\bullet$ or $[A,Y]^\bullet \subseteq [A,X]^\bullet$ holds.
\end{proposition}

\begin{proof}
If $\r(X,Y) < \infty$, then the triangle inequality implies that $\r(A,Y) < \infty$.  So assume that $\r(X,Y) = \infty$, thus also that $\r(Y,X) < \infty$.  Lemma \ref{lemma:Neighbors} then shows that $X^-$ lies in both $[A,X]^\bullet$ and $[Y,X]^\bullet$ such that Proposition \ref{Proposition:InbetweenObjects} then yields that $Y \in [A,X]^\bullet$.  We find that $\r(A,Y) < \infty$.

To prove the second claim, note that by Lemma \ref{lemma:Neighbors} we may assume, possibly by interchanging $X$ and $Y$, that $[X,Y]^\bullet$ is a thread containing $Y^-$.  Lemma \ref{lemma:Neighbors} shows that $Y^- \in [A,Y]^\bullet$.  Proposition \ref{Proposition:InbetweenObjects} then yields that $X \in [A,Y]^\bullet$.  From Corollary \ref{corollary:HalfConvex} we obtain that $[A,X]^\bullet \subseteq [A,Y]^\bullet$.
\end{proof}

\begin{proposition}\label{proposition:Anchors}
Let $\QQ$ be a hereditary section with nonthread objects.
\begin{enumerate}
\item Every ray has an anchor.
\item A ray is uniquely determined by its anchor and a (unique) direct successor of the anchor which lies $\r$-in-between the anchor and the ray.
\item Only finitely many rays can share an anchor.
\end{enumerate}
\end{proposition}

\begin{proof}
Let $\RR$ be a ray.  For any ray object $X$ in $\RR$, Lemma \ref{lemma:FewRoughsInIntervals} gives a nonthread object $A_X$ which is the only nonthread object in $[A_X,X]_\QQ^\bullet$.  To show that $\RR$ has an anchor, we wish to show all these objects $A_X$ coincide.

Let $X,Y \in \RR$.  By Lemma \ref{lemma:Neighbors}, we may assume that $Y^- \in [X,Y]^\bullet$.  Since by $\r(A_Y,Y) = \infty$, Lemma \ref{lemma:Neighbors} also yields that $Y^- \in [A_Y,X]^\bullet$ such that Proposition \ref{Proposition:InbetweenObjects} shows that $X \in [A_Y,Y]^\bullet$.  Again using Lemma \ref{lemma:Neighbors}, we see that also $X^- \in [A_Y,Y]^\bullet$.

Consider a co-light cone $\QQ^X$ centered on $X$.  Denote by $A_X'$ and $A_Y'$ the objects in $\QQ^Y$ lying in the $\tau$-orbits of $A_X$ and $A_Y$, respectively. The embedding $[A_X',X] \to \QQ^X$ has a left adjoint $l:[\QQ^X,A_X'] \to \QQ^X$.  Since $\Hom(A_X',X^-) \not= 0$, we know that $X$ is not the only direct summand of $l(A_X')$.  Lemma \ref{lemma:ThreadsAndAdjoints} then shows that either $A_X' \in [A_Y',X]$, or $\Hom(A_X',A_Y') \not= 0$ and thus $A_Y' \in [A_X',X]$.

Proposition \ref{Proposition:SameOrbits} shows that $A_X \in [A_Y,X]^\bullet \subseteq [A_Y,X]^\bullet$ or $A_Y \in [A_X,X]^\bullet$.  We conclude that $A_X \cong A_Y$.  This shows that $\RR$ has an anchor.

For the second point, let $\RR$ be a ray with anchor $A$.  Let $X \in \ind \RR$.  It follows from Lemma \ref{lemma:Neighbors} that at least one direct successor of $A$ lies in $[A,X]^\bullet$. Let $A_1$ and $A_2$ be two direct successors of $A$, both lying in $[A,X]^\bullet$ where $X \in \ind \RR$.  We see that both $[A_1,X]^\bullet$ and $[A_2,X]^\bullet$ are threads.  Again using Lemma \ref{lemma:Neighbors}, we see that $X^-$ lies in both of them.  Applying Proposition \ref{Proposition:InbetweenObjects} yields that $A_1 = A_2$.

Let $\RR_1$ and $\RR_2$ be two rays with the same anchor $A$, and let $X_1 \in \ind \RR_1$ and $X_2 \in \ind \RR_2$.  It follows from \ref{lemma:Neighbors} that there are neighbors $A_1,A_2 \in \ind \QQ$ of $A$ such that $A_1 \in [A,X_1]^\bullet_\QQ$ and $A_2 \in [A,X_2]^\bullet_\QQ.$  If $A_1 = A_2$, then it follows from Proposition \ref{Proposition:InbetweenObjects} that $X_1$ and $X_2$ lie on the same thread.  We conclude that the number of rays which have $A$ as an anchor is limited by the number of direct neighbors of $A$.
\end{proof}

Because of Proposition \ref{proposition:Anchors}, it will sometimes be more convenient to assume a hereditary section has a nonthread object and hence every ray and coray has an anchor and a coanchor, respectively.  The following examples show that this is not necessarily the case.

\begin{example}
The category of projectives $\QQ$ of $\AA$ from Example \ref{example:LargerTilt} forms a hereditary section in $\Db \AA$ which has no nonthread objects.  The hereditary section $\QQ_\TT$ constructed in the aforementioned exercise has nonthread objects and satisfies $\bZ \QQ = \bZ \QQ_\TT$.
\end{example}

\begin{example}\label{example:NoNonthreads}
Let $Q$ be the thread quiver
$$\xymatrix@1{x \ar@<3pt>@{..>}[r] \ar@<-3pt>[r] & y}$$
as in Example \ref{example:DifferentDistances}.  Let $X$ be the indecomposable projective object corresponding to the vertex $x$.  In $\Db \mod kQ$ there is a unique hereditary section which has no nonthread objects as given in Figure \ref{fig:NoNonThreads}.
\begin{figure}[tb]
	\centering
		\includegraphics[width=0.50\textwidth]{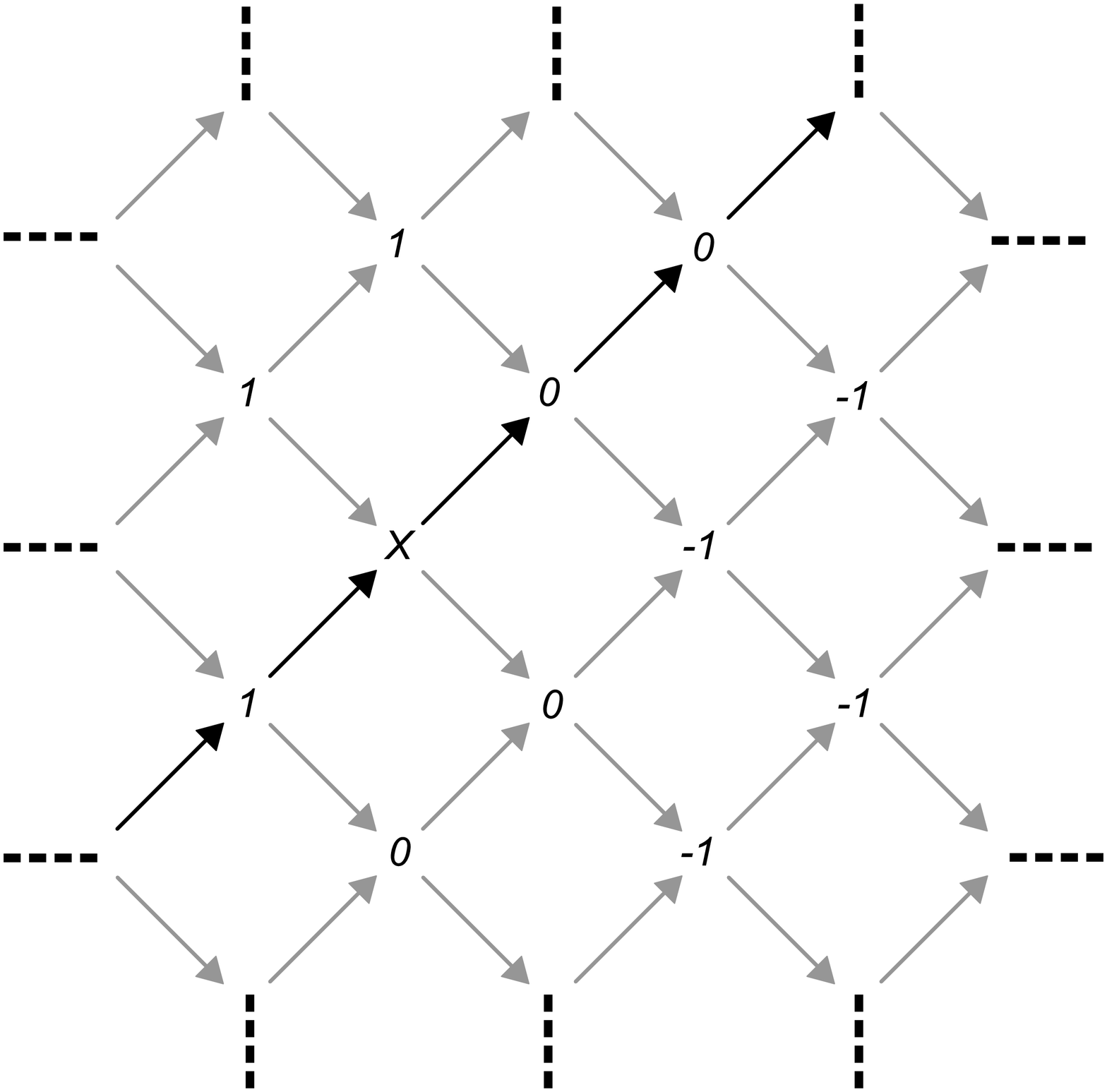}
	\caption{A hereditary section which has no nonthread objects from Example \ref{example:NoNonthreads} (the arrows between the indecomposable objects in the chosen hereditary section have been drawn in black).  Every vertex is labeled with $\r(X,-)$.}
	\label{fig:NoNonThreads}
\end{figure}
\end{example}

The next proposition shows we can always replace a hereditary section without nonthread objects by a hereditary section which has nonthread objects.

\begin{proposition}\label{proposition:HasRoughObject}
Let $\QQ$ be a hereditary section in $\Db \AA$.  There is a hereditary section $\QQ'$ with $\bZ \QQ = \bZ \QQ'$ having at least one nonthread object.
\end{proposition}

\begin{proof}
Assume $\QQ$ has only thread objects.  We define a new hereditary section $\QQ_1$ generated by the indecomposables $\t^{-m_Y} Y$ where $Y \in \QQ$ and $m_Y = \min(\r(X,Y),1)$.  To prove this is indeed a hereditary section, it suffices to show that $\r(\t^{-m_Y} Y,\t^{-m_Z} Z) \geq 0$ for all $Y,Z \in \ind \QQ$ (see Corollary \ref{corollary:TauTilting}).  We have
\begin{eqnarray*}
\r(\t^{-m_Y} Y,\t^{-m_Z} Z) &=& \r(Y,Z) +m_Y - m_Z \\
&\geq& \r(Y,Z) + \r(X,Y) - m_Z \\
&\geq& \r(X,Z) - m_Z \geq 0
\end{eqnarray*}

Let $X \in \ind \QQ$ and $X^- \in \ind \QQ$ be the unique direct predecessor of $X$ in $\QQ$.  Since $X$ is directed, we have $\r(X,X^-) \not= 0$.  This shows that $m_{X^-} = 1$.  We deduce that $X$ is a nonthread object in $\QQ'$.
\end{proof}

\begin{proposition}\label{proposition:CountableRays}
Let $\bZ \QQ$ be connected.  Then $\QQ$ has only countably many rays and corays.
\end{proposition}

\begin{proof}
If $\QQ$ does not have any nonthread objects, then this statement is trivial.  Otherwise, this follows from Proposition \ref{proposition:Anchors} together with Proposition \ref{proposition:NonthreadStructure}.
\end{proof}
\section{Categories generated by $\bZ \QQ$}\label{section:GeneratedBybZQQ}

Let $\AA$ be a $k$-linear abelian Ext-finite category with Serre duality and let $\QQ$ be a nonzero hereditary section in $\CC=\Db \AA$.  We will consider the case where $\Db \AA$ is generated by $\bZ \QQ$, thus the smallest thick triangulated subcategory of $\Db \AA$ containing $\bZ \QQ$ is $\Db \AA$ itself.  This is, for example, the case if $\QQ$ is the standard hereditary section when $\AA$ is generated by projectives or --more generally-- by preprojective objects.

If $\QQ$ satisfies the condition (*), then Theorem \ref{theorem:MainTilting} shows that $\Db \AA \cong \rep Q$ for a strongly locally finite thread quiver $Q$.  We are thus interested in the case where $\QQ$ does not satisfy condition (*).  In this case, there will always be rays and/or corays (see Observation \ref{observation:NeedsRays}).  We will replace $\QQ$ by another another (larger) hereditary section $\QQ'$ such that $\bZ \QQ \subset \bZ \QQ'$ and such that $\QQ'$ does not have rays nor corays.  We may then apply Theorem \ref{theorem:MainTilting} to obtain our main result: $\AA$ is derived equivalent to $\rep Q'$ for a locally finite thread quiver $Q'$.

We start by defining marks and comarks which will be used to enlarge $\QQ$.

\subsection{Marks and comarks}
While an anchor should indicate ``the beginning'' of a ray (how it is attached to the nonthread objects), a mark should indicate ``the direction'' or ``the ending'' of a ray.

Let $\QQ$ be a hereditary section with at least one nonthread object and let $\RR$ be a ray with anchor $A$.  Let $B$ be the direct successor of $A$ lying in $[A,X]^\bullet_\QQ$ (see Proposition \ref{proposition:Anchors}).  The \emph{mark} $M$ of $\RR$ is defined to be the cone of the irreducible morphism $\tau B \to A$.  Since $\Hom(A,\tau B[1]) = 0$, one easily verifies using Proposition \ref{proposition:RingelTriangles} (cf. \cite[Corollary 1.4]{HappelZacharia08} or \cite[Lemma 7.6]{vanRoosmalen06}) that $M$ is indecomposable.

Dually, one defines comarks for corays.

\begin{example}
Examples \ref{example:NotStar} and \ref{example:Halfopen2} are obtained starting from a thread quiver $Q = \xymatrix@1{x \ar@{..>}[r]^{\PP} & z}$.  In these examples, the simple projective $P_x \in \rep_k Q$ also lies in the given hereditary section and is the anchor of the unique thread.  The mark is then given by $P_z[0] \in \Db \rep_k Q$.
\end{example}

\begin{lemma}\label{lemma:MarksBound}
Let $\QQ$ be a hereditary section with at least one nonthread object.  Let $\RR$ be a ray with anchor $A$ and mark $M$.  For every $X \in \ind \RR$, we have $\r(A,X) \geq 0$ and  $\r(X,M) < \infty$.
\end{lemma}

\begin{proof}
The first statement is included in the definition of an anchor.  For the second statement, let $B$ be the unique direct successor of $A$ in $[A,X]^\bullet$ and write $n = \r(A,X)$.  From $\r(A,\t^{-n}X) = 0$ and $\r(B,\t^{-n}X) = \r(A,\t^{-n}X) - \r(A,B) = 0$ follows (using Proposition \ref{proposition:LeftmostComposition})
\begin{eqnarray*}
\dim \Hom(\t^{-n}X,A[1]) = &\dim \Hom(A,\t^{-n+1}X)& = 0, \\ 
\dim\Hom(\t^{-n} X,\t B[1]) = &\dim\Hom(B,\t^{-n}X)& \not= 0.
\end{eqnarray*}
Applying $\Hom(\t^{-n} X,-)$ to the triangle $A \to M \to \t B[1] \to A[1]$ then yields that $\Hom(\t^{-n} X,M) \not= 0$ such that $\r(\t^{-n} X,M) \leq 0$ and hence also $\r(X,M) \leq n < \infty$.
\end{proof}

In general, we may have $\r(X,M) = -\infty$ as the following example shows.

\begin{example}\label{example:AddingMarksToD}
Let $Q$ be the thread quiver
$$\xymatrix@R=5pt{ &&& \cdot \\x \ar@{..>}[rr]^1 && y \ar[ru] \ar[rd] \\ &&& \cdot}$$
Denote by $P_x$ and $P_y$ the indecomposable projective objects in $\rep Q$ corresponding to the vertices $x$ and $y$ respectively.

The category $\Db \rep Q$ is sketched in the upper part of Figure \ref{fig:NotGenerated}.  Let $\QQ$ be the standard hereditary section, and let $\QQ'$ be the the hereditary section spanned by all indecomposables of $\QQ$ which do not lie in the Auslander-Reiten component of $P_y$.  Let $\HH$ be a heart corresponding to $\QQ'$ as in the lower part of Figure \ref{fig:NotGenerated}.
\begin{figure}
	\centering
		\includegraphics[width=.60\textwidth]{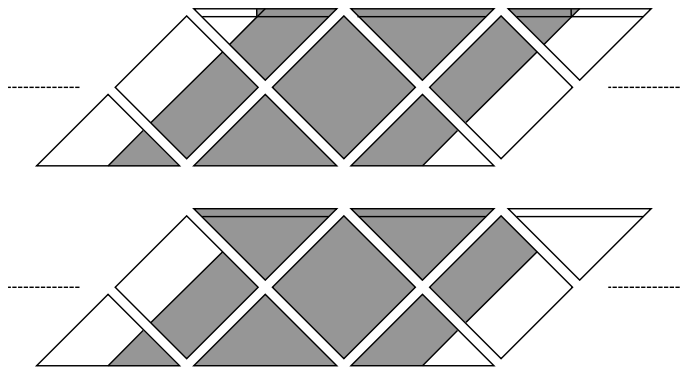}
	\caption{Illustration of Example \ref{example:AddingMarksToD}.}
	\label{fig:NotGenerated}
\end{figure}
The hereditary section $\QQ'$ has a unique ray $\RR$ with anchor the projective indecomposable $P_x$ and as mark the injective indecomposable $\bS P_x = I_x$.  Note that $\r(I_x,P_x) = -\infty$.
\end{example}

\subsection{Enlarging hereditary sections}

This subsection is devoted to proving Proposition \ref{proposition:ExtendingSections} below where we will extend $\bZ \QQ$ to a subcategory of the form $\bZ \QQ'$ which does satisfy condition (*).  It will be the main step in the proof of Theorem \ref{theorem:GeneratedByZQQ}.  The proof will follow from Lemma \ref{lemma:NiceGeneratingSet} and Proposition \ref{proposition:BiggerTT}.

\begin{proposition}\label{proposition:ExtendingSections}
Let $\QQ$ be a hereditary section with nonthread objects.  Assume $\bZ \QQ$ generates $\Db \AA$.  Then there is a hereditary section $\QQ'$ such that $\bZ \QQ'$ generates $\Db \AA$ and satisfies condition (*).
\end{proposition}

We start by giving a lemma we will use to find the required ``larger'' hereditary section.

\begin{lemma}\label{lemma:NiceGeneratingSet}
Let $\TT \subset \ind \Db \AA$ be a countable set such that $\r(T_i,T_j) \not= -\infty$, for all $T_i,T_j \in \TT$.  Then there is a hereditary section $\QQ'$ in $\Db \AA$ such that $\TT \subset \bZ \QQ'$ and $\bZ \QQ'$ satisfies condition (*).
\end{lemma}

\begin{proof}
Possibly by taking different objects from the same $\t$-orbits, we may assume that $\r(T_i,T_j)\geq 0$ for all $T_i,T_j \in \TT$.  Consider the full replete additive category $\CC$ spanned by all indecomposables $X$ of $\Db \AA$ such that $\d(\TT,X) \in \bZ$.  We choose a new set $\TT'$ from objects in $\CC$ as in Lemma \ref{lemma:ChooseTT} (the proof carries over from $\bZ \QQ'$ to $\CC$) and define a full additive subcategory $\QQ'$ of $\CC$ such that $\r (\TT, X) = \left\lfloor \frac{\d (\TT, X)}{2}\right\rfloor$.  As in Proposition \ref{proposition:IsHereditarySection} one shows $\QQ'$ is a hereditary section in $\Db \AA$.
\end{proof}

We are thus reduced to finding a suitable set $\TT$ satisfying the properties of the previous lemma.  We show that, if $\QQ$ has nonthread objects, one can choose $\TT$ to be the set of all nonthread objects, all marks, and all comarks.

For this, we first consider the following situation.  Let $R_1$ be the thread quiver whose underlying quiver is an $A_\infty$-quiver with zig-zag orientation (the zigs and zags can have arbitrary finite length) and where the thread arrows all point away from the base point $x$ (see for example Figure \ref{figure:R1Quiver}).  Since $R_1$ is a strongly locally finite thread quiver, we know that $\rep R_1$ is a hereditary category with Serre duality (Theorem \ref{corollary:ShDualizing}).   Denote by $\RR_1$ the standard hereditary section in $\Db \rep R_1$ and consider the light cone $\RR_2$ centered on $P_x$.

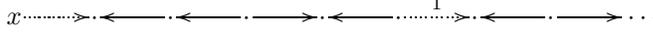
\begin{figure}
\label{figure:R1Quiver}
$$\xymatrix@1{x \ar@{..>}[r]&\cdot&\cdot\ar[l]&\cdot\ar[l]\ar[r]&\cdot&\cdot\ar[l]\ar@{..>}[r]^1&\cdot&\cdot\ar[l]\ar[r]&\cdots&}$$
\caption{An example of the thread quiver $R_1$}
\end{figure}

It is readily verified (for example by using the classification provided in \cite{vanRoosmalen06}) that $\RR_2$ is a semi-hereditary dualizing $k$-variety given by a thread quiver $R_2 = \xymatrix@1{x \ar@{..>}[r]^{\PP} & z}$ for a suitably chosen linearly ordered poset $\PP$.  Furthermore, the categories $\rep R_1$ and $\rep R_2$ are derived equivalent.

Let $Q$ be a strongly locally finite thread quiver and let $x$ be any vertex.  We construct the thread quivers $Q_1$ and $Q_2$ by identifying the vertex $x \in Q$ with the base point of $R_1$ and $R_2$, respectively.

\begin{proposition}\label{proposition:ReplacingZigZag}
In the above situation, $Q_1$ and $Q_2$ are strongly locally finite thread quivers and $\Db \rep Q_1 \cong \Db \rep Q_2$ as triangulated categories.
\end{proposition}

\begin{proof}
It is an easy observation that both thread quivers are indeed strongly locally finite such that $\rep Q_{1}$ and $\rep Q_{2}$ have Serre duality.  Denote by $\QQ_1$ and $\QQ_2$ the categories of projectives of $\rep Q_{1}$ and $\rep Q_{2}$, respectively.

There is an obvious fully faithful functor from $\QQ_1$ into $\bZ \QQ_2$ lifting to a fully faithful exact functor $G:\Db \rep \QQ_1 \to \Db \rep \QQ_2$ which commutes with the Serre functor.  To show $G$ is an equivalence it suffices to show the essential image of $G$ contains $\QQ_2$.

Let $R_2= \xymatrix{x \ar@{..>}[r]^{\PP} & z}$ be the thread quiver attached to the thread quiver $Q$ to form $Q_2$.  We show that the projective object $P_z$ associated to $z$ lies in the essential image of $G$.  Let $P_{y} \in \ind \rep Q_1$ be the unique direct successor of $x$ lying in $[P_x, P_z]$ and let $I_y$ be the corresponding injective; there is a short exact sequence $0 \to P_x \to P_z \to I_y \to 0$.  Both $P_x$ and $I_y$ lie in $\bZ \QQ_1$, and hence it follows that $P_z$ lies in the essential image of $G$.  Since $G$ commutes with the Serre functor, the entire Auslander-Reiten component containing $P_z$ also lies in the essential image of $G$.
\end{proof}

This situation will be encountered in the following case.  Let $\QQ$ be a hereditary section with nonthread objects and satisfying condition (*).  Let $\RR$ be a ray in $\QQ$ with anchor $A$.  By Theorem \ref{theorem:MainTilting} there is a hereditary section $\QQ'$ which is a semi-hereditary dualizing $k$-variety such that $\bZ \QQ = \bZ \QQ'$.  We will assume, for ease of notation, that $A \in \ind \QQ'$.

Denote by $\RR'$ the full additive subcategory of $\QQ'$ lying in $\bZ \RR$.  For every $Y' \in \ind \RR'$, there is a full additive subcategory $[A,Y']^\bullet_{\QQ'}$ in $\QQ'$; denote by $\RR'_1$ the smallest full replete additive subcategory containing all these $[A,Y']^\bullet_{\QQ'}$.  Thus $\RR'_1$ contains $A$, every object $Y'$, and every object $\r$-in-between $A$ and $Y'$.

Let $Q'$ be a thread quiver of $\QQ'$ containing a vertex $a$ corresponding to $A$ and denote by $R'_1$ the subquiver corresponding with $\RR'_1$.

\begin{lemma}
The thread quiver $R'_1$ described above is an $A_\infty$-quiver with zig-zag orientation (the zigs and zags can have arbitrary finite length) and where the thread arrows all point away from the base point $a'$.
\end{lemma}

\begin{proof}
For all $Y \in \RR$, we know that every $X \in \ind ]A,Y]^\bullet_{\QQ}$ is a thread object.  Hence for every $X' \in \ind ]A,Y']^\bullet_{\QQ'}$ (with $Y' \in \RR'_1$), we have only the following possibilities: either $X'$ is a source or a sink with exactly two direct successors or predecessors respectively, or $X'$ is a thread object.

Let $R'_{1,r}$ be the underlying quiver of $R'_1$, thus forgetting the distinction between thread arrows and regular arrows.  A straightforward argument shows that $R'_{1,r}$ is either an $A_\infty$-quiver or an $\tilde{A}_n$-quiver (with arbitrary orientation).

Since Proposition \ref{proposition:Anchors} yields there is a unique neighbor of $A$ lying in $[A,Y]^\bullet_\QQ$, for all $Y \in \RR$, we know $R'_{1,r}$ is an $A_\infty$-quiver.  Furthermore, since for every $X \in \ind \RR$, we have $\r(A,X) < \infty$, all thread arrows in $R_1$ point away from the base point $a$.
\end{proof}

We may now apply Proposition \ref{proposition:ReplacingZigZag} to replace the subquiver $R'_1$ of $Q'$ by $R'_2 = \xymatrix@1{a \ar@{..>}[r]^{\PP} & z}$ and obtain thus a quiver $Q'_2$.  Let $\QQ'_2$ be the hereditary section in $\Db \rep Q'_2$ given by the projective representations in $\rep Q'_2$.  Using the derived equivalence in Proposition \ref{proposition:ReplacingZigZag}, we will interpret $\QQ'_2$ as a hereditary section in $\Db \AA$.  Note that $A,B \in \ind \bZ \QQ'_2$.  It is now readily verified that $P_z$ corresponds to the mark $M$.  Hence we have shown the following proposition.

\begin{proposition}\label{proposition:MarksInbZQQ}
Let $\QQ$ be a hereditary section with nonthread objects.  Assume furthermore that $\bZ \QQ$ satisfies condition (*).  Let $\RR$ be a ray with mark $X$, then there is a hereditary section $\QQ'$ such that $X \in \bZ \QQ'$ and $\bZ \QQ \subset \bZ \QQ'$. 
\end{proposition}

We obtain following result.

\begin{proposition}\label{proposition:BiggerTT}
Let $\QQ$ be a hereditary section with nonthread objects, such that $\bZ \QQ$ generates $\Db \AA$.  Let $\TT$ be the set of all nonthread objects, all marks of rays, and all comarks of corays.  Then $\TT$ is countable, and $\r(X,Y) \not= -\infty$, for all $X,Y \in \TT$.  Furthermore, for every $A \in \bZ \QQ$, we have $\d(\TT,A) \in \bZ$.
\end{proposition}

\begin{proof}
Seeking a contradiction, let $X,Y \in \TT$ such that $\r(X,Y) = -\infty$.  Thus, for every $n \in \bZ$ there is a path $X = Z_{n,0} \to Z_{n,1} \to \cdots \to Z_{n,k_n} = \t^{n} Y$ in $\Db \AA$.

With an $\SS \subseteq \bZ \QQ$ we associate the hereditary section $\QQ_\SS$ generated by $A \in \QQ_\SS$ with $\d(\SS,A) < \infty$.  If $\SS$ is countable, then $\bZ \QQ_\SS$ satisfies condition (*) by construction.

We will choose $\SS$ such that $\bZ \QQ_\SS$ generates all objects $Z_{n,i}$ as above.  Furthermore, if $X$ or $Y$ are a mark or comark of a ray or coray, we will assume $\SS$ has at least one ray or coray object from the corresponding ray or coray.  It is clear that $\SS$ can be chosen countably.

Proposition \ref{proposition:MarksInbZQQ} yields that there is a hereditary section $\QQ'_\SS$  such that $X,Y \in \bZ \QQ'_\SS$ and $\bZ \QQ_\SS \subseteq \bZ \QQ'_\SS$, thus $\r(X,Y)\not= -\infty$.  A contradiction, since we had assumed that all objects $Z_{n,i}$ were generated by $\bZ \QQ'_\SS$.

We are left to proving that for every $A \in \bZ \QQ$, we have $\d(\TT,A) \in \bZ$.  Since $\TT$ contains all nonthread objects of $\QQ$, this should be clear when $A$ is not a ray or coray object.  In case $A$ is a ray or a coray object, then this follows easily from Proposition \ref{lemma:MarksBound} or its dual, respectively.
\end{proof}

We can now use Proposition \ref{proposition:ExtendingSections} to prove our main theorem.

\begin{theorem}\label{theorem:GeneratedByZQQ}
Let $\AA$ be a $k$-linear abelian hereditary Ext-finite category with Serre duality and assume $\Db \AA$ is generated by $\bZ \QQ$ as triangulated category, then $\AA$ is derived equivalent to $\rep Q$ for a strongly locally finite thread quiver $Q$.
\end{theorem}

\begin{proof}
This follows easily from Proposition \ref{proposition:ExtendingSections} and Theorem \ref{theorem:MainTilting}.
\end{proof}

\begin{corollary}
Let $\AA$ be a $k$-linear abelian hereditary Ext-finite category with Serre duality which is generated by preprojective objects.  Then $\AA$ is derived equivalent to $\rep Q$ where $Q$ is a strongly locally finite thread quiver.
\end{corollary}

\providecommand{\bysame}{\leavevmode\hbox to3em{\hrulefill}\thinspace}
\providecommand{\MR}{\relax\ifhmode\unskip\space\fi MR }
\providecommand{\MRhref}[2]{%
  \href{http://www.ams.org/mathscinet-getitem?mr=#1}{#2}
}
\providecommand{\href}[2]{#2}

\end{document}